\documentclass[11pt,a4 paper,oneside, reqno]{amsart}
\usepackage[utf8]{inputenc}
\usepackage[OT1]{fontenc}
\usepackage[english]{babel}
\usepackage{dsfont}
\usepackage{amsfonts}
\usepackage{amsmath}
\usepackage{bbm}
\usepackage{wasysym}
\usepackage{amsthm}
\usepackage{amssymb}
\usepackage{float}
\usepackage{textcomp}
\usepackage{xcolor}
\usepackage{graphicx}
\usepackage{fancybox}
\usepackage{fancyhdr}
\usepackage{mathrsfs}
\usepackage{tikz}
\usetikzlibrary{arrows}
\usetikzlibrary{calc}
\usepackage{centernot}
\usepackage{listings}
\usepackage{faktor}
\usepackage{bm}
\usepackage{setspace}
\usepackage{hyperref}
\usepackage{newlfont}
\usepackage{geometry}
\usepackage{ccaption}
\usepackage{tipa}
\usepackage{units}
\usepackage[autostyle,italian=guillemets]{csquotes}
\usepackage{comment}

\usepackage{guit}
\usepackage{tikz-cd}
\usepackage{enumerate}

\everymath{\displaystyle}

\theoremstyle{definition}
\newtheorem{Def}{Definition}[section]

\newtheorem{ese}[Def]{Examples}
\newtheorem{as}[Def]{Assumption}

\theoremstyle{remark}
\newtheorem{obs}[Def]{Remark}
\theoremstyle{plain}
\newtheorem{prop}[Def]{Proposition}

\newtheorem{lema}[Def]{Lemma}

\newtheorem{cor}[Def]{Corollary}

\newtheorem{teo}[Def]{Theorem}

\newcommand{\bo}{\mathbf}
\newcommand{\A}{{\mathcal A}}
\newcommand{\B}{{\mathcal B}}
\newcommand{\C}{{\mathcal C}}
\newcommand{\D}{{\mathcal D}}
\newcommand{\E}{{\mathcal E}}

\newcommand{\G}{{\mathcal G}}

\newcommand{\I}{{\mathcal I}}

\newcommand{\K}{{\mathcal K}}
\renewcommand{\L}{{\mathcal L}}
\newcommand{\M}{{\mathcal M}}

\renewcommand{\P}{{\mathcal P}}
\newcommand{\Q}{{\mathcal Q}}

\renewcommand{\S}{{\mathcal S}}
\newcommand{\T}{{\mathcal T}}

\newcommand{\V}{{\mathcal V}}

\newcommand{\mt}{\mathscr}
\newcommand{\tx}{\textnormal}

\newcommand{\colim}{\operatornamewithlimits{colim}}

\makeatletter
\newcommand{\changeoperator}[1]{%
	\csletcs{#1@saved}{#1@}%
	\csdef{#1@}{\changed@operator{#1}}%
}
\newcommand{\changed@operator}[1]{%
	\mathop{%
		\mathchoice{\textstyle\csuse{#1@saved}}
		{\csuse{#1@saved}}
		{\csuse{#1@saved}}
		{\csuse{#1@saved}}%
	}%
}
\makeatother

\makeatletter
\def\@tocline#1#2#3#4#5#6#7{\relax
	\ifnum #1>\c@tocdepth 
	\else
	\par \addpenalty\@secpenalty\addvspace{#2}%
	\begingroup \hyphenpenalty\@M
	\@ifempty{#4}{%
		\@tempdima\csname r@tocindent\number#1\endcsname\relax
	}{%
		\@tempdima#4\relax
	}%
	\parindent\z@ \leftskip#3\relax \advance\leftskip\@tempdima\relax
	\rightskip\@pnumwidth plus4em \parfillskip-\@pnumwidth
	#5\leavevmode\hskip-\@tempdima
	\ifcase #1
	\or\or \hskip 1em \or \hskip 2em \else \hskip 3em \fi%
	#6\nobreak\relax
	\hfill\hbox to\@pnumwidth{\@tocpagenum{#7}}\par
	\nobreak
	\endgroup
	\fi}
\makeatother

\changeoperator{sum}
\changeoperator{prod}
\changeoperator{coprod}

\title{Virtual concepts in the theory of accessible categories}
\author{Stephen Lack and Giacomo Tendas}
\address{School of Mathematical and Physical Sciences, Macquarie University NSW 2109, 
	Australia}
\email{steve.lack@mq.edu.au}
\address{School of Mathematical and Physical Sciences, Macquarie University NSW 2109, 
	Australia}
\email{giacomo.tendas@mq.edu.au}
\date{\today}
\thanks{The first-named author acknowledges with gratitude the support of an Australian Research Council Discovery Project DP190102432. The second-named author gratefully acknowledges the support of an International Macquarie University Research Excellence Scholarship 2018115/20191524.}

\begin{document}
	
\begin{abstract}
	We provide a new characterization of enriched accessible categories by introducing the two new notions of {\em virtual reflectivity} and {\em virtual orthogonality} as a generalization of the usual reflectivity and orthogonality conditions for locally presentable categories. The word {\em virtual} refers to the fact that the reflectivity and orthogonality conditions are given in the free completion of the $\V$-category involved under small limits, instead of the $\V$-category itself. In this way we hope to provide a clearer understanding of the theory as well as a useful way of recognizing accessible $\V$-categories. In the last section we prove that the 2-category of accessible $\V$-categories, accessible $\V$-functors, and $\V$-natural transformations has all flexible limits.
	
\end{abstract}	
	
\maketitle
	
\tableofcontents

\section{Introduction}

The importance and usefulness of enriched accessible categories have been recognized recently by many authors in different areas of category theory: for instance in the 2-categorical context \cite{Bou2021:articolo}, in the additive \cite{Pre11:libro} and simplicial \cite{BLV:articolo} ones, and in the world of $\infty$-cosmoi \cite{BL21accessible}.
The theory of enriched locally presentable categories was developed in 1982 by Kelly \cite{Kel82:libro}, but the theory of enriched accessible categories, first introduced in the late 1990s \cite{BQ96:articolo}, is much less developed. 

The reason for this delay with regards to enriched accessibility might be found in the following two points:\begin{enumerate}\setlength\itemsep{0.25em}
	\item There is not a unique way of defining what an accessible $\V$-category is; in fact two different notions have already been considered in the literature. Those that we simply call {\em accessible} arise as free cocompletions of small $\V$-categories under $\alpha$-flat colimits for some $\alpha$ \cite{BQR98}; the others, that we call {\em conically accessible}, arise instead as free cocompletions of small $\V$-categories under (conical) $\alpha$-filtered colimits for some $\alpha$. 
	\item The theory of accessible categories is much harder to address than that of the locally presentable ones, even in the ordinary setting where many of the proofs are very $\bo{Set}$-based and thus not suitable for an enriched generalization. This results in a scarcity of characterization theorems, most importantly in the recognition of accessible $\V$-categories. 
\end{enumerate}
While we addressed (1) in \cite{LT21:articolo}, where we gave sufficient conditions on $\V$ for the two notions on the one hand to coincide, or on the other hand to coincide up to Cauchy completeness; we aim to address (2) in the present paper. Here we focus mostly on the accessible $\V$-categories of \cite{BQR98}, nonetheless we compare these with (and give results on) the conically accessible ones as well.

The characterization theorem we prove seems to be new even in the ordinary setting, although some of the machinery we use was already considered, for $\V=\bo{Set}$, by Guitart and Lair \cite{GL81:articolo}. The idea is to generalize the notion of orthogonality and reflectivity (central in the characterization of locally presentable categories) to those of {\em virtual orthogonality} and {\em virtual reflectivity} as described below. Here the word ``virtual'' refers to the fact that the objects and morphisms we are considering lie not in a given category $\A$, but rather in its free completion $\P^\dagger\A$ under limits.

There are virtual notions of: left adjoint, reflective subcategory, cocomplete, and orthogonality class. 
To introduce the notion of virtual left adjoint, recall that a functor $F\colon\A\to\K$ has a left adjoint if $\K(X,F-)$ is representable for any $X\in\K$; we say instead that $F$ has a {\em virtual left adjoint} if $\K(X,F-)$ is a small functor for any $X\in\K$, that is if $\K(X,F-)\in\P^\dagger\A$ for all $X\in\K$. If $F$ is fully faithful we say that $\A$ is {\em virtually reflective} in $\K$. Clearly every reflective subcategory is virtually reflective, but the converse is not true. When $\V=\bo{Set}$ and $\A$ and $\K$ are both accessible, then virtual reflectivity is equivalent to the inclusion functor satisfying the solution-set condition (Corollary~\ref{virt-sset}).

Regarding virtual colimits, we know that, given a weight $M\colon\C^{op}\to\V$ and a $\V$-functor $H\colon\C\to\A$, the colimit $M*H$ exists in $\A$ if the $\V$-functor $[\C^{op},\V](M,\A(H,-))$ is representable. Relaxing that condition, we say that the {\em virtual colimit} of $H$ weighted by $M$ exists in $\A$ if $[\C^{op},\V](M,\A(H,-))$ is a small $\V$-functor. Since every representable $\V$-functor is small, any cocomplete $\V$-category has virtual colimits; less trivial is the fact that also every accessible $\V$-category has them (Proposition~\ref{accvirtcoco}). When $\V=\bo{Set}$ a category $\A$ has virtual colimits if and only if it is pre-cocomplete in the sense of Freyd \cite{Fre72:articolo} (Proposition~\ref{petty-lucis-small}).

An object $A$ of a $\V$-category $\K$ is said to be orthogonal with respect to a morphism $f\colon X\to Y$ in $\K$ if the map $\K(f,A)$ is an isomorphism in $\V$. Then the notion of {\em virtual orthogonality} arises exactly in the same way with the only difference being that the morphism $f$ now can be chosen to be of the form $f\colon ZX\to Y$ with $X\in\K$ and $Y\in\P^\dagger\K$, where $Z\colon\K\hookrightarrow\P^\dagger\K$ is the inclusion. Thus we will say that $A\in\K$ is orthogonal with respect to a morphism $f\colon ZX\to Y$ in $\P^\dagger\K$ if the map $\P^\dagger\K(f,A)$ is an isomorphism in $\V$. We call {\em virtual orthogonality class} a full subcategory of $\K$ which arises as a collection of objects which are virtually orthogonal with respect to a small set of morphisms as above.

Then our main theorem goes as follows:

\begin{teo}[\ref{accessible-embedded}]
	For an accessible $\V$-category $\K$ and a fully faithful inclusion $\A\hookrightarrow\K$, the following are equivalent:\begin{enumerate}\setlength\itemsep{0.25em}
		\item $\A$ is accessible and accessibly embedded;
		\item $\A$ is accessibly embedded and virtually reflective;
		\item $\A$ is a virtual orthogonality class.
	\end{enumerate}
\end{teo}

Even though our proofs still rely on some inevitable $\bo{Set}$-based conditions, such as the raising of the accessibility index, we believe they provide a more formal and cleaner approach to the theory with respect to some of the concepts studied in the past for ordinary accessible categories. We study the relationship between these and our notions in Section~\ref{cre+p}.

In the context of conically accessible $\V$-categories a similar characterization can be given. This shows that, in contrast to accessibility, conical accessibility of a subcategory can be recognized at the level of the underlying ordinary category.

\begin{teo}[\ref{con-accessible-embedded}]
	For a conically accessible $\V$-category $\K$ and a fully faithful inclusion $\A\hookrightarrow\K$, the following are equivalent:\begin{enumerate}\setlength\itemsep{0.25em}
		\item $\A$ is conically accessible and conically accessibly embedded;
		\item $\A_0$ is accessible and accessibly embedded in $\K_0$;
		\item $\A_0$ is accessibly embedded and virtually reflective in $\K_0$;
		\item $\A_0$ is a virtual orthogonality class in $\K_0$.
	\end{enumerate}
\end{teo}

We begin the paper by recalling in Section~\ref{back} some results on ordinary accessible categories and setting the notation for enriched categories. In Section~\ref{sectionacc}, following the ideas of \cite{ABLR02:articolo}, we introduce the notion of $\Phi$-accessible $\V$-categories for a sound class of weights $\Phi$ and prove some basic results which generalize those of the accessible $\V$-categories of \cite{BQR98}. Then we recall some basic facts about conically accessible $\V$-categories. Finally we compare the two notions, for a general base $\V$, by giving sufficient conditions for a conically accessible $\V$-category to be accessible, and by showing that if $\A$ is $\alpha$-accessible then it is also conically $\beta$-accessible for any $\beta$ sharply greater than $\alpha$.

In Section~\ref{mainresults-acc} we prove the main results of this paper. We begin by giving a new proof of the fact that a $\V$-category is accessible if and only if it is sketchable (which was first shown in \cite{BQR98}) and then we prove the characterization theorems above. In the last subsection we compare the virtual concepts with those of cone-reflectivity and cone-injectivity, obtaining some of the results in \cite{AR94:libro} as a consequence of our theorem. 

To conclude the paper, in Section~\ref{limits} we prove that the 2-category of accessible $\V$-categories, accessible $\V$-functors, and $\V$-natural transformations is closed in $\V\tx{-}\bo{CAT}$ under all small flexible limits, and it therefore has all pseudo and bilimits as well. The same holds if we replace accessibility by conical accessibility.

\section{Background notions}\label{back}

\subsection{Results on ordinary accessible categories}

In this section we recall the only result on ordinary accessible categories that will be used throughout this paper; that is about raising the index of accessibility:

\begin{Def}
	Given two regular cardinals $\alpha$ and $\beta$, we say that $\alpha$ is {\em sharply less} than $\beta$, and write $\alpha\lhd\beta$, if $\alpha<\beta$ and for every $\alpha$-filtered category $\C$ and any $\beta$-small $\D\subseteq\C$ there exists a $\beta$-small and $\alpha$-filtered $\E$ with $\D\subseteq\E\subseteq\C$.
\end{Def}

This is one of many equivalent set-theoretic definitions of the sharply less than relation. Equivalently, one could consider in the definition above the case when $\C$ is just an $\alpha$-directed poset (this notion was considered in \cite[Theorem~2.11(iv)]{AR94:libro}). Another set-theoretic characterization is as follows: $\alpha\lhd\beta$ if and only if $\alpha<\beta$ and for every set $X$ of cardinality less than $\beta$, the partially ordered set $\P_\alpha(X)$, of all subsets of $X$ with cardinality less than $\alpha$, has a final subset of cardinality less than $\beta$. This is how the sharply less relation was originally defined in \cite[2.3.1]{MP89:libro}.

\begin{obs}
	For any small set of regular cardinals $\{\alpha_i\}_{i\in I}$ there are arbitrarily large regular cardinals $\beta$ for which $\beta\rhd\alpha_i$ holds for all $i\in I$ \cite[Example~2.13(6)]{AR94:libro}. 
\end{obs}

Before stating the next theorem let us fix some notation. For a regular cardinal $\alpha$ and a category $\C$  we denote by $\alpha\tx{-Ind}(\C)$ the free cocompletion of $\C$ under $\alpha$-filtered colimits; this can be described as the full subcategory of $[\C^{op},\bo{Set}]$ spanned by the $\alpha$-flat functors. For given regular cardinals $\alpha<\beta$ we denote by $\C_{\nicefrac{\beta}{\alpha}}$ the free cocompletion of $\C$ under $\beta$-small $\alpha$-filtered colimits; this is a one-step completion by \cite[Corollary~4.13]{chen2021sift}. Let $J\colon\C\hookrightarrow \C_{\nicefrac{\beta}{\alpha}}$ be the inclusion; then since $\C_{\nicefrac{\beta}{\alpha}}^{op}$ is the free completion of $\C^{op}$ under $\beta$-small $\alpha$-cofiltered limits, it follows that 
$$ (-\circ J)\colon \nicefrac{\beta}{\alpha}\tx{-Cont}(\C_{\nicefrac{\beta}{\alpha}}^{op},\bo{Set})\to[\C^{op},\bo{Set}] $$
is an equivalence of categories with inverse $\tx{Ran}_{J^{op}}$, where $\nicefrac{\beta}{\alpha}\tx{-Cont}[\C_{\nicefrac{\beta}{\alpha}}^{op},\bo{Set}]$ denotes the full subcategory of $[\C_{\nicefrac{\beta}{\alpha}}^{op},\bo{Set}]$ spanned by those functors that preserve $\beta$-small $\alpha$-filtered limits. Consider now a representable $\C_{\nicefrac{\beta}{\alpha}}(-,C)$, since $C$ is a $J$-absolute $\beta$-small $\alpha$-filtered colimit of elements from $\C$, it follows that $\C_{\nicefrac{\beta}{\alpha}}(J-,C)\cong \colim_i\C(-,C_i)$ is a $\beta$-small $\alpha$-filtered colimit of representables, and in particular $\alpha$-flat. Moreover, since pre-composition by $J$ preserves $\beta$-filtered colimits (being cocontinuous) and these are also $\alpha$-filtered, it follows that we have an induced functor
$$ (-\circ J) \colon\beta\tx{-Flat}(\C_{\nicefrac{\beta}{\alpha}}^{op},\bo{Set}) \hookrightarrow\alpha\tx{-Flat}(\C^{op},\bo{Set}) $$
which is fully faithful since every $\beta$-flat functor is $\nicefrac{\beta}{\alpha}$-continuous (it preserves all existing $\beta$-small limits).

We are now ready to state and prove the following theorem, which can be seen as an expansion of \cite[Theorem~2.11]{AR94:libro}.

\begin{teo}\label{raising}
	Let $\alpha<\beta$ be two regular cardinals and $J\colon\C\hookrightarrow\C_{\beta/\alpha}$ be as above; the following are equivalent:
	\begin{enumerate}\setlength\itemsep{0.25em}
		\item $\alpha\lhd\beta$;
		\item if $\C$ is an $\alpha$-filtered category then $\C_{\nicefrac{\beta}{\alpha}}$ is $\beta$-filtered;
		\item if $M\colon\C^{op}\to\bo{Set}$ is an $\alpha$-flat functor then $\tx{Ran}_{J^{op}}M\colon \C_{\nicefrac{\beta}{\alpha}}^{op}\to\bo{Set}$ is $\beta$-flat;
		\item each $\alpha$-flat $M\colon\C^{op}\to\bo{Set}$ is the restriction of some $\beta$-flat $N\colon\C_{\nicefrac{\beta}{\alpha}}^{op}\to\bo{Set}$;
		\item $\alpha\tx{-Ind}(\C)\simeq\beta\tx{-Ind}(\C_{\nicefrac{\beta}{\alpha}})$ for any small $\C$;
		\item every $\alpha$-accessible category is $\beta$-accessible.
	\end{enumerate}
\end{teo}
\begin{proof}
	$(1)\Rightarrow (2)$. Let $\C$ be an $\alpha$-filtered category and $\{X_i\}_{i\in I}$ be a $\beta$-small family of objects in $\C_{\nicefrac{\beta}{\alpha}}$. For each $i\in I$ fix a $\beta$-small diagram $H_i\colon\D_i\to\C$ whose colimit in $\C_{\nicefrac{\beta}{\alpha}}$ is $X_i$. By $(1)$ we can consider a $\beta$-small and $\alpha$-filtered $\E\subseteq\C$ which contains the images of all the $H_i$'s. Let $X$ be the colimit of the inclusion of $\E$ in $\C_{\nicefrac{\beta}{\alpha}}$; then by construction we have an induced arrow $X_i\to X$ for any $i\in I$, as desired.
	
	Consider now a $\beta$-small family of parallel arrows $\{f_i\colon X\to Y\}_{i\in I}$; then there are $\beta$-small categories $\D_i$ and diagrams $H_i\colon\D_i\to\C^\mathbbm{2}$ whose colimits in $(\C_{\nicefrac{\beta}{\alpha}})^\mathbbm 2$ are the $f_i$'s. Again, by $(1)$ we can find a $\beta$-small and $\alpha$-filtered $\E$ in $\C$ which contains the images of all the $H_i$'s. It follows then that the colimit of $\E$ in $\C_{\nicefrac{\beta}{\alpha}}$ comes with a cocone for the family $\{f_i\}_{i\in I}$. 
	
	$(2)\Rightarrow (3)$. Consider an $\alpha$-flat $M\colon\C^{op}\to\bo{Set}$ and its right Kan extension $N\colon\C_{\nicefrac{\beta}{\alpha}}^{op}\to\bo{Set}$ along $J^{op}$. Using the fact that $\C_{\nicefrac{\beta}{\alpha}}$ is the free cocompletion of $\C$ under $\beta$-small $\alpha$-filtered colimits, one can show that $\tx{El}(N)$ is the free cocompletion of $\tx{El}(M)$ under the same kind of colimits, so that $\tx{El}(N)=\tx{El}(M)_{\nicefrac{\beta}{\alpha}}$. Now, since $M$ is $\alpha$-flat, then $\tx{El}(M)$ is $\alpha$-filtered, and thus $\tx{El}(N)$ is $\beta$-filtered by $(2)$.
	
	$(3)\Rightarrow (4)$. This is trivial assuming $(3)$ since $M$ is always the restriction of its right Kan extension along $J^{op}$.
	
	$(4)\Rightarrow (5)$. Thanks to $(4)$ and the comments above the Theorem, $(-\circ J)$ induces an equivalence between the $\alpha$-flat functors out of $\C^{op}$ and the $\beta$-flat functors out of $\C_{\nicefrac{\beta}{\alpha}}^{op}$. Thus $(5)$ follows at once.
	
	$(5)\Rightarrow (6)$ is trivial, while $(6)\Rightarrow(1)$ can be shown as in the proof of \cite[Theorem~2.11]{AR94:libro}.

\end{proof}

A direct consequence of the equivalence between $(1)$ and $(6)$ is that the sharply less than relation is transitive.

\begin{obs}
	Point $(5)$ of the Theorem can be restated as follows. Given the monads $P=\beta\tx{-Ind}(-)$, that freely adds $\beta$-filtered colimits, and $T=(-)_{\nicefrac{\beta}{\alpha}}$, that freely adds $\beta$-small $\alpha$-filtered colimits, the composite $PT$ is still a monad and coincides with $\alpha\tx{-Ind}(-)$. This results in a distributive law from $T$ to $P$.
\end{obs}

Moreover:

\begin{cor}{\cite[Proposition~2.3.11]{MP89:libro}}\label{raising2}
	For an $\alpha$-accessible category $\A$ and regular cardinals $\alpha\lhd\beta$, there is an equivalence $\A_\beta\simeq(\A_\alpha)_{\nicefrac{\beta}{\alpha}}$ and hence every $\beta$-presentable object of $\A$ is a $\beta$-small $\alpha$-filtered colimit of $\alpha$-presentable objects.
\end{cor}
\begin{proof}
	The fact that $\A_\beta\simeq(\A_\alpha)_{\nicefrac{\beta}{\alpha}}$ is a direct consequence of condition $(5)$ from the Theorem above and of the fact that $\alpha\tx{-Ind}(\C)_\alpha\simeq\C$ for any $\alpha$ and any Cauchy complete $\C$. The last assertion is a consequence of \cite[Corollary~4.13]{chen2021sift}.
\end{proof}

A direct consequence of these results is the following:

\begin{cor}\cite[Theorem~2.19]{AR94:libro}\label{preserving-presentables}
	Given an accessible functor $F\colon\A\to\B$ between ordinary accessible categories, there exists $\alpha$ such that $F$ preserves the $\beta$-presentable objects for each $\beta\rhd\alpha$.
\end{cor}
\begin{proof}
	Let $\alpha_0$ be such that $\A$, $\B$, and $F$ are $\alpha_0$-accessible and let $\alpha\rhd\alpha_0$ be such that $F(\A_{\alpha_0})\subseteq \B_\alpha$. Consider now $\beta\rhd\alpha$; by transitivity of the sharply less than relation and Corollary~\ref{raising2} each object of $\A_\beta$ is a $\beta$-small $\alpha_0$-filtered colimit of objects from $\A_{\alpha_0}$; since $F$ preserves $\alpha_0$-filtered colimits it follows that each object of $F(\A_\beta)$ is a $\beta$-small ($\alpha_0$-filtered) colimit of objects from $\B_\alpha$, and thus is still $\beta$-presentable in $\B$.
\end{proof}

\vspace{6pt}

\subsection{Background on enriched categories}

We now fix a base of enrichment $\V=(\V_0,\otimes,I)$ which is locally presentable and symmetric monoidal closed; in particular $\V$ will be locally $\alpha_0$-presentable as a closed category (\cite{Kel82:articolo}) for some regular cardinal $\alpha_0$. From now on every cardinal taken into consideration will be regular and greater than $\alpha_0$.
	
We follow the notations of \cite{Kel82:libro}, with the only change that ``indexed'' colimits are here called ``weighted'', as is now standard. In particular, given a $\V$-category $\A$ we denote by $\A_0$ its underlying ordinary category; similarly if $F\colon \A\to\B$ is a $\V$-functor we denote by $F_0\colon \A_0\to\B_0$ the corresponding ordinary functor underlying $F$. For any ordinary category $\K$ we denote by $\K_\V$ the free $\V$-category over $\K$. Our $\V$-categories are allowed to have a large set of objects, unless specified otherwise.
	
We call {\em weight} a $\V$-functor $ M \colon \C^{op}\to\V$ with small domain. Given such a weight $ M $ and a $\V$-functor $H\colon \C\to\A$, we denote by $ M *H$ (if it exists) the colimit of $H$ weighted by $ M $; dually weighted limits are denoted by $\{ N ,K\}$ for $ N \colon \C\to\V$ and $K\colon \C\to\A$. Conical limits and colimits are special cases of weighted ones; they coincide with those weighted by $\Delta I\colon \B_\V^{op}\to\V$ for some ordinary category $\B$. The conical colimit of a $\V$-functor $T_{\V}\colon \B_\V\to\A$, if it exists, will also be the ordinary colimit of the transpose $T\colon \B\to\A_0$ in $\A_0$ (but the converse is not generally true). 
	
\begin{Def}[\cite{Kel82:articolo}]
	We say that a weight $ M \colon \C^{op}\to\V$ is {\em $\alpha$-small} if $\C$ has less than $\alpha$ objects, $\C(C,D)\in\V_{\alpha}$ for any $C,D\in\C$, and $ M (C)\in\V_\alpha$ for any $C\in\C$. An $\alpha$-small (weighted) limit is one taken along an $\alpha$-small weight. We say that a $\V$-category $\C$ is $\alpha$-complete if it has all $\alpha$-small limits; we say that a $\V$-functor is $\alpha$-continuous if it preserves all $\alpha$-small limits.
\end{Def}
	
Both conical $\alpha$-small limits and powers by $\alpha$-presentable objects are examples of $\alpha$-small limits and together are enough to generate all $\alpha$-small weighted limits \cite[Section~4]{Kel82:articolo}.

\begin{Def}
	We say that a weight $ M \colon\C^{op}\to\V$ is $\alpha$-flat if its left Kan extension $\tx{Lan}_Y M \colon[\C,\V]\to\V$ along the Yoneda embedding is $\alpha$-continuous. By {\em $\alpha$-flat colimits} we mean those weighted by an $\alpha$-flat weight.
\end{Def}

Since $\tx{Lan}_Y M\cong M*-$ it follows that $M$ is $\alpha$-flat if and only if $M$-weighted colimits commute in $\V$ with $\alpha$-small colimits. Note that every conical $\alpha$-filtered colimit is $\alpha$-flat and that the $\alpha$-flat $\V$-functors are closed in $[\C^{op},\V]$ under $\alpha$-flat colimits.

\begin{prop}[\cite{Kel82:articolo}]\label{flat-char}
	Let $ M \colon\C^{op}\to\V$ be a weight; the following are equivalent:\begin{enumerate}\setlength\itemsep{0.25em}
		\item $ M $ is $\alpha$-flat;
		\item $ M $ is an $\alpha$-flat colimit of representables.
	\end{enumerate}
	If $\C$ is $\alpha$-complete they are further equivalent to:\begin{enumerate}\setlength\itemsep{0.25em}
		\item[(3)] $ M $ is $\alpha$-continuous;
		\item[(4)] $ M $ is a conical $\alpha$-filtered colimit of representables. 
	\end{enumerate}
	In that case the following isomorphism holds
	$$  M \cong\tx{colim} (\tx{El}( M _I)_{\V} \stackrel{\pi_\V}{\longrightarrow} \C \stackrel{Y}{\longrightarrow} [\C^{op},\V])$$
	where $ M _I=\V_0(I, M _0-)$ and $\tx{El}( M _I)$ is $\alpha$-filtered.
\end{prop}

\section{Accessible $\V$-categories}\label{sectionacc}

\subsection{$\Phi$-accessible $\V$-categories}\label{con-flat}

In this section we introduce the main notion of accessibility that we consider in the present paper, this can be seen as a generalization of that of \cite{BQR98} to the ``sound'' context of \cite{ABLR02:articolo}. Most of this section will depend on Assumption~\ref{soundclass} below.

\begin{Def}
	Let $\Phi$ be a class of weights. We say that a weight $M\colon\C^{op}\to\V$ is {\em $\Phi$-flat} if its left Kan extension $\tx{Lan}_YM\colon[\C,\V]\to\V$ along the Yoneda embedding is $\Phi$-continuous. We call {$\Phi$-flat colimits} those weighted by a $\Phi$-flat weight and denote by $\Phi\tx{-Flat}(\C^{op},\V)$ the full subcategory of $[\C^{op},\V]$ spanned by the $\Phi$-flat weights.
\end{Def}

\begin{Def}
	A class of weights $\Phi$ is called {\em weakly sound} if every $\Phi$-continuous $\V$-functor $ M \colon\C^{op}\to\V$ (from a $\Phi$-cocomplete $\C$) is $\Phi$-flat.\\
	A class $\Phi$ is called {\em sound} if, for any $ M\colon\C^{op}\to\V$, whenever $ M *-$ preserves $\Phi$-limits of representables, then $ M $ is $\Phi$-flat.
\end{Def}

Of course if $\Phi$ is sound then it is weakly sound, but the converse doesn't always hold: see \cite[Remark~2.6]{ABLR02:articolo}. However, as we are going to see below, the converse does hold when the class of weights $\Phi$ is {\em pre-saturated}, meaning that for any small $\C$ the free cocompletion $\Phi\C$ of $\C$ under $\Phi$-colimits is a one-step closure in $[\C^{op},\V]$. In other words $\Phi$ is pre-saturated if, for any $\V$-category $\C$, every object of $\Phi\C$ is a $\Phi$-colimit of objects from $\C$.

We first need to prove a lemma which is a generalization of \cite[Lemma~2.7]{LT21:articolo}.

\begin{lema}\label{flat-restriction}
	Let $\Phi$ be a class of weights, $J\colon \B\to\C$ be a $\V$-functor, and $  M  \colon \B^{op}\to\V$ a weight; then:\begin{enumerate}\setlength\itemsep{0.25em}
		\item if $  M  $ is $\Phi$-flat then $\tx{Lan}_{J^{op}}  M $ is;
		\item if $J$ is fully faithful and $\tx{Lan}_{J^{op}} M  $ is $\Phi$-flat then $  M  $ is $\Phi$-flat as well.
	\end{enumerate}
\end{lema}
\begin{proof}
    The proof is exactly the same as that of \cite[Lemma~2.7]{LT21:articolo}, just replace $\alpha$-flatness and $\alpha$-continuity by $\Phi$-flatness and $\Phi$-continuity everywhere.
\end{proof}

\begin{prop}\label{weaklyiffstrongsound}
	Let $\Phi$ be a pre-saturated class of weights. Then $\Phi$ is weakly sound if and only if it is sound.
\end{prop}
\begin{proof}
	One direction is clear. Suppose then that $\Phi$ is pre-saturated and weakly sound, and let $ M :\C^{op}\to\V$ be such that $ M *-:[\C,\V]\to\V$ preserves $\Phi$-limits of representables; we need to prove that $ M $ is $\Phi$-flat.
	
	Let $\Phi^\dagger(\C^{op})$ be the free completion of $\C^{op}$ under $\Phi$-limits and $J:\C^{op}\hookrightarrow \Phi^\dagger(\C^{op})$ be the inclusion, note that equivalently $\Phi^\dagger(\C^{op})=\Phi(\C)^{op}$ is the opposite of the free cocompletion of $\C$ under $\Phi$-colimits. Consider $ M ':=\tx{Lan}_JM$; by Lemma \ref{flat-restriction} it follows that $ M $ is $\Phi$-flat if and only if $ M '$ is $\Phi$-flat. Moreover, since $\Phi$ is weakly sound, $ M '$ is $\Phi$-flat if and only if it is $\Phi$-continuous.
	
	Thus it will suffice to prove that $ M '$ is $\Phi$-continuous. Note first that $M'$ preserves $\Phi$-limits of diagrams landing in $\C^{op}$: take $N :\D\to\V$ in $\Phi$ and $H:\D\to \C^{op}$ then
	\begin{align}
			 M '\{ N ,JH\}&\cong (\tx{Lan}_JM)\{ N ,JH\}\nonumber \\
			&\cong  M -*\ \Phi^\dagger(\C^{op})(J-,\{ N ,JH\})\nonumber \\
			&\cong  M -*\{ N \square,\C(-,H\square)\}\nonumber \\
			&\cong \{ N \square, M -*\C(-,H\square)\}\\
			&\cong \{ N , M \circ H\}\nonumber \\
			&\cong \{ N , M '\circ JH\}\nonumber
	\end{align}
	where $(1)$ follows from the fact that $ M *-$ preserves $\Phi$-limits of representables.\\
	Now, since $\Phi$ is pre-saturated, every object of $\Phi^\dagger(\C^{op})$ is a $J$-absolute $\Phi$-limit of a diagram landing in $\C^{op}$; therefore $ M '$ preserves all the $J$-absolute limits of a chosen codensity presentation of $J$. By \cite[Proposition~2.2]{Day77:articolo} we then have $ M '\cong \tx{Ran}_JM$ (so that left and right Kan extensions coincide). But $\Phi^\dagger(\C^{op})$ is the free completion of $\C^{op}$ under $\Phi$-limits, therefore the functor $\tx{Ran}_JM$ is $\Phi$-continuous by the universal property of such completion. This means that $ M '\cong \tx{Ran}_JM$ is $\Phi$-continuous, and hence $\Phi$-flat.
\end{proof}

Recall from \cite{KS05:articolo} that a class of weights $\Phi$ is called {\em essentially small} if for every small $\V$-category $\C$ the free cocompletion $\Phi\C$ is still small. 

\begin{as}\label{soundclass}
    From now on we fix an essentially small and weakly sound class $\Phi$.
\end{as}

\begin{ese}\label{examples} The following are examples of essentially small weakly sound classes. In some examples the class $\Phi$ is described as a class of indexing categories, these should be understood as the corresponding classes of conical weights.
	\begin{enumerate}\setlength\itemsep{0.25em}
		\item $\Phi=\emptyset$. Then any weight is $\Phi$-flat, and thus $\Phi$ is sound since every presheaf is a weighted colimit of representables.
		\item $\V$ locally $\alpha$-presentable as a closed category, $\Phi$ the class of $\alpha$-small weights. Then $\Phi$-flat weights are the usual $\alpha$-flat $\V$-functors. This is weakly sound by Proposition~\ref{flat-char}, and hence sound being pre-saturated.
		\item $\V$ symmetric monoidal closed finitary quasivariety, $\Phi$ the class of weights for finite products and finitely presentable projective powers. This is weakly sound by \cite[Theorem~5.8]{LR11NotionsOL}.
		\item $\V$ cartesian closed, $\Phi$ the class for finite products. Then a weight $M$ is $\Phi$-flat if and only if $\tx{Lan}_\Delta M \cong M \!\times\! M $, where $\Delta\colon\C\to\C\times\C$ is the diagonal and $M \!\times\! M\colon\C\times\C\to\V$ is defined by $(M \!\times\! M)(A,B)=MA\times MB$. This is weakly sound thanks to \cite[Lemma~2.3]{KL93FfKe:articolo}, and hence, being pre-saturated, is also sound.
		\item $\V=\bo{Set}$, $\Phi=\{\emptyset\}$. Then $\Phi$-flat colimits are generated by connected colimits. Soundness is discussed in \cite{ABLR02:articolo}. 
		\item $\V=\bo{Set}$, $\Phi$ the class of finite connected categories. Then $\Phi$-flat colimits are generated by coproducts and filtered colimits. Soundness is discussed in \cite{ABLR02:articolo}. 
		\item $\V=\bo{Set}$, $\Phi$ the class of finite non empty categories. Then a functor is $\Phi$-flat if and only if its category of elements is empty or filtered. It follows easily from this that the class is sound.
		\item $\V=\bo{Set}$, $\Phi$ the class of finite discrete non empty categories. Then a functor is $\Phi$-flat if and only if its category of elements is empty or sifted. Soundness follows as above by replacing filtered with sifted.
		\item $\V=\bo{Cat}$, $\Phi$ the class for finite connected 2-limits, meaning the class generated by finite connected conical limits and powers by finite connected categories. Then $\Phi$-flat colimits contain filtered colimits and coproducts. Soundness can be proven as in (2) and (3).
		\item $\V=([\C^{op},\bo{Set}],\otimes,I)$ with the representables closed under the monoidal structure, $\Phi$ the class for powers by representables. Then conical colimits are $\Phi$-flat and every $\Phi$-continuous $\V$-functor can be written as a conical colimit of representables (argue as in the case of $\alpha$-continuous $\V$-functors). It follows that $\Phi$ is weakly sound.
	\end{enumerate}
\end{ese}

\begin{Def}
	Let $\A$ be a $\V$-category with $\Phi$-flat colimits. A $\V$-functor $F\colon\A\to\B$ is called {\em $\Phi$-accessible} if it preserves $\Phi$-flat colimits; an object $A$ of $\A$ is called {\em $\Phi$-presentable} if $\A(A,-)\colon\A\to\V$ is $\Phi$-accessible. We denote by $\A_\Phi$ the full subcategory of $\A$ spanned by the $\Phi$-presentable objects. 
\end{Def}

\begin{Def}	
	We say that $\A$ is {\em $\Phi$-accessible} if it has $\Phi$-flat colimits and there exists a small $\C\subseteq\A_\Phi$ such that every object of $\A$ can be written as a $\Phi$-flat colimit of objects from $\C$.
\end{Def}

When $\Phi=\Phi_\alpha$, for some $\alpha$, we simply say that $\A$ is {\em $\alpha$-accessible} instead of $\Phi_\alpha$-accessible, and if this is so for some $\alpha$, we say that $\A$ is {\em accessible}. This agrees with the definition in \cite{BQR98}.

The first results we can prove about $\Phi$-accessibility are a standard generalization of the ordinary ones.

\begin{prop}\label{A_alpha}
	Let $\A$ be a $\Phi$-accessible $\V$-category and $H\colon\A_{\Phi}\hookrightarrow\A$ be the inclusion; then:\begin{enumerate}\setlength\itemsep{0.25em}
		\item $\A_{\Phi}$ is Cauchy complete and closed in $\A$ under all existing $\Phi$-colimits;
		\item $\A_{\Phi}$ is (essentially) small;
		\item every $A\in\A$ can be expressed as the $\Phi$-flat colimit:
		$$ A\cong \A(H-,A)*H. $$
	\end{enumerate}
\end{prop}
\begin{proof}
	$(1)$. If $A\in\A$ is a $\Phi$-colimit of $\alpha$-presentable objects in $\A$, then $\A(A,-)$ is a $\Phi$-limit of $\V$-functors preserving $\Phi$-flat colimits. Since these colimits commute in $\V$ with $\Phi$-limits, $\A(A,-)$ still preserves $\Phi$-flat colimits and hence $A\in\A_\Phi$. The same argument applies to Cauchy colimits.
	
	$(2)$. Let $H\colon\C\hookrightarrow\A$ be a small full subcategory of $\Phi$-presentable objects witnessing the fact that $\A$ is $\Phi$-accessible. We shall show that $\A_\Phi$ is the Cauchy completion $\Q(\C)$ of $\C$, and so is essentially small by \cite{Joh1989:articolo}. Since $\C$ is contained in $\A_\Phi$ and $\A_\Phi$ is Cauchy complete, $\Q(\C)$ is contained in $\A_\Phi$. For the converse, suppose that $A\in\A_\Phi$ and let $Z\colon\A\to[\C^{op},\V]$ be the induced fully faithful functor with $ZX=\A(H-,X)$. Note now that an element of $[\C^{op},\V]$ is small projective (an absolute weight) if and only if it lies in $\Q(\C)$; moreover, by \cite[Proposition~6.14]{KS05:articolo}, a $\V$-functor $F\in[\C^{op},\V]$ is small projective if and only if $[\C^{op},\V](F,-)$ preserves the colimit $F*Y$. Then
	\begin{equation*}
		\begin{split}
			[\C^{op},\V](ZA,ZA*Y)&\cong [\C^{op},\V](ZA,Z(ZA*H))\\
			&\cong \A(A,ZA*H)\\
			&\cong ZA*\A(A,H-)\\
			&\cong ZA*[\C^{op},\V](ZA,Y-)
		\end{split}
	\end{equation*}
	where we used that $ZH\cong Y$. It follows that $ZA$ is small projective and hence $\A_\Phi=\Q(\C)$ is the Cauchy completion of $\C$, and hence is essentially small by \cite{Joh1989:articolo}.
	
	$(3)$. Given $A\in\A$, by hypothesis we can write $A\cong M *HF$ for some $\Phi$-flat $M\colon\C^{op}\to\V$ and $F\colon\C\to\A_\Phi$. Then $\A(H-,A)\cong M \square*\A_{\Phi}(-,F\square)$ is a $\Phi$-flat colimit of representables, and hence $\Phi$-flat; moreover 
	\begin{equation*}
		\begin{split}
			\A(H-,A)*H&\cong  M \square*(\A_{\Phi}(-,F\square)*H-)\\
			&\cong M*F\\
			&\cong A
		\end{split}
	\end{equation*}
	as desired.
\end{proof}

\begin{prop}\label{acc=flat}
	The following are equivalent for a $\V$-category $\A$: \begin{enumerate}\setlength\itemsep{0.25em}
		\item $\A$ is $\Phi$-accessible;
		\item $\A$ is the free cocompletion of a small $\V$-category $\C$ under $\Phi$-flat colimits;
		\item $\A\simeq\Phi\tx{-}\tx{Flat}(\C^{op},\V)$ for some small $\C$.
	\end{enumerate}
	In both $(2)$ and $(3)$ the $\V$-category $\C$ can be chosen to be $\A_\Phi$; moreover, if $\C$ is Cauchy complete, then $\Phi\tx{-}\tx{Flat}(\C^{op},\V)_\Phi\simeq \C$.
\end{prop}
\begin{proof}
	The equivalence $(1)\Leftrightarrow (2)$ and the fact that $\C$ can be chosen to be $\A_\Phi$ are a direct consequence of \cite[Proposition~4.3]{KS05:articolo}.
	
	$(2)\Rightarrow (3)$. Let $H\colon\C\to\A$ be the inclusion. By \cite[Proposition~4.3]{KS05:articolo} $\C$ is made of $\Phi$-presentable objects in $\A$ and is small; therefore the induced $\V$-functor $J:=\A(H,1)\colon \A\to[\C^{op},\V]$ is fully faithful and preserves $\Phi$-flat colimits. 
	Given any $A\in\A$ we can write it as a $\Phi$-flat colimit $A\cong M *HK$ of objects from $\C$; thus $JA\cong  M  *JHK\cong  M *YK$ is a $\Phi$-flat colimit of representables, and hence a $\Phi$-flat $\V$-functor. Vice versa, given a $\Phi$-flat $ M \colon\C^{op}\to\V$ we know that $J( M *H)\cong  M *Y\cong  M $. As a consequence $\A\simeq\Phi\tx{-}\tx{Flat}(\C^{op},\V)$ as claimed and $\C$ can be chosen to be $\A_\Phi$ (since that was true for the second point).
	
	$(3)\Rightarrow(1)$. Let $\A=\Phi\tx{-}\tx{Flat}(\C^{op},\V)$, $J\colon\A\to[\C^{op},\V]$ be the inclusion, and $H\colon\C\hookrightarrow\A$ be the full subcategory of $\A$ spanned by the representables, so that $JH=Y$ is the Yoneda embedding. Then $$\A(HC,-)\cong[\C^{op},\V](YC,J-)\cong\tx{ev}_C\circ J$$ for any $C\in\C$; thus $\A(HC,-)$ preserves $\Phi$-flat colimits since $J$ does and $\tx{ev}_C$ is cocontinuous. It follows that the representable functors in $\A$ are $\Phi$-presentable objects. Moreover we can write every $ M \in\A$ as a $\Phi$-flat colimit of representables as $ M \cong M *Y$. This shows that $\A$ is $\Phi$-accessible. 
	
	Regarding the last statement, as a consequence of the proof of Proposition~\ref{A_alpha}, we know that $\Phi\tx{-}\tx{Flat}(\C^{op},\V)_\Phi\simeq\Q(\C)$. Thus, if $\C$ is Cauchy complete then $\C=\Q(\C)$, and thus $\Phi\tx{-}\tx{Flat}(\C^{op},\V)_\Phi\simeq \C$.  
\end{proof}

\begin{obs}\label{phi-acc-funct}
	It follows, from the universal property of free cocompletions, that a $\V$-functor $F\colon\A\to\B$ out of a $\Phi$-accessible $\V$-category $\A$ is $\Phi$-accessible if and only if it is the left Kan extension of its restriction to $\A_\Phi$.
\end{obs}

We will show in Section~\ref{accvscon} that, in the context of $\alpha$-accessible categories, the index of accessibility can again be raised with the sharply less than relation, as in the ordinary context. Moreover we will see that the underlying ordinary category of an accessible $\V$-category is again accessible.

\vspace{6pt}

\subsection{Conically accessible $\V$-categories}

In this section we consider a different notion of enriched accessibility which involves only (conical) $\alpha$-filtered colimits. This goes as follows:

\begin{Def}
	Let $\alpha$ be a regular cardinal and $\A$ be a $\V$-category with $\alpha$-filtered colimits. A $\V$-functor $F\colon\A\to\B$ is called {\em conically $\alpha$-accessible} if it preserves $\alpha$-filtered colimits; an object $A$ of $\A$ is called {\em conically $\alpha$-presentable} if $\A(A,-)\colon\A\to\V$ is conically $\alpha$-accessible. We denote by $\A_\alpha^c$ the full subcategory of $\A$ spanned by the $\alpha$-presentable objects. 
\end{Def}	

\begin{Def}	
	We say that $\A$ is {\em conically $\alpha$-accessible} if it has $\alpha$-filtered colimits and there exists $\C\subseteq\A_\alpha^c$ small such that every object of $\A$ can be written as an $\alpha$-filtered colimit of objects from $\C$. 
\end{Def}

The next results are then immediate consequences of the definitions:

\begin{prop}\label{conicallyaccproperties}
	Let $\A$ be a conically $\alpha$-accessible $\V$-category and $H\colon\A_{\alpha}^c\hookrightarrow\A$ be the inclusion; then:\begin{enumerate}\setlength\itemsep{0.25em}
		\item $\A_{\alpha}^c$ is closed in $\A$ under existing $\alpha$-small colimits;
		\item $\A_{\alpha}^c$ is (essentially) small;
		\item every $A\in\A$ can be expressed as the $\alpha$-filtered colimit:
		$$ A\cong \colim\left((\A_{\alpha}^c)_0/A\stackrel{\pi}{\longrightarrow}\A_{\alpha}\stackrel{}{\longrightarrow}\A \right); $$
		\item $\A_0$ is $\alpha$-accessible and $(\A_{\alpha}^c)_0=(\A_0)_{\alpha}$.
	\end{enumerate}
\end{prop}
\begin{proof}
	$(1)$. Let $A=\colim A_i$ be an $\alpha$-small colimit of conically $\alpha$-presentable objects; then $\A(A,-)\cong \lim \A(A_i,-)$ is an $\alpha$-small limit of functors which preserve $\alpha$-filtered colimits. Since $\alpha$-small limits commute with $\alpha$-filtered colimits, $\A(A,-)$ preserves $\alpha$-filtered colimits as well and $A\in\A_\alpha^c$.
	
	$(2)$. Let $\C\subseteq\A_\alpha^c$ be small and generate $\A$ under $\alpha$-filtered colimits, and let $A\in\A_\alpha^c$. As a consequence $A\cong\colim C_i$ is an $\alpha$-filtered colimit of elements from $\C$. Since by hypothesis $\A(A,-)$ preserves $\alpha$-filtered colimits, it follows that the identity map $\tx{Id}_A$ factors through some $C_i$. Therefore $A$ is a split subobject of some object of $\C$ and hence $(\A_\alpha^c)_0$ is the Cauchy completion of $\C_0$. Since $\C$ is small, $\A_\alpha^c$ is as well.
	
	$(3)$. Let $\C$ be as in $(2)$ above and $A$ be any object of $A$. Then $A\cong\colim HF$ where $F\colon\D\to\A_\alpha^c$ is a diagram with $\alpha$-filtered domain. Then the colimit cocone of $A$ induces a final functor $K\colon\D\to (\A_{\alpha})_0/A$ such that $\pi\circ K=H$. It follows then that $(\A_{\alpha})_0/A$ is $\alpha$-filtered (this is a general fact about final functors with $\alpha$-filtered domain, see \cite[Remark~2.8]{LT21:articolo} for instance) and the colimit of its projection on $\A$ is $A$.
	
	$(4)$. Given any object $A\in\A$, since the unit of $\V$ is $\alpha$-presentable, if $\A(A,-)$ preserves $\alpha$-filtered colimits then so does $\A_0(A,-)=\V_0(I,\A(A,-)_0)$; therefore $(\A_\alpha^c)_0\subseteq (\A_0)_\alpha$. As a consequence $(\A_\alpha^c)_0$ is a small full subcategory of $\A_0$ made of $\alpha$-presentable objects, and generates $\A$ under $\alpha$-filtered colimits. This implies that $\A_0$ is $\alpha$-accessible as an ordinary category. Finally, given $A\in(\A_0)_\alpha$, arguing as in $(2)$ we can write $A$ as a split subobject of some $B\in(\A_\alpha)_0$; since $(\A_\alpha)_0$ is closed under split subobjects in $\A_0$ it follows that $(\A_{\alpha}^c)_0=(\A_0)_{\alpha}$.
	
\end{proof}

\begin{prop}
	The following are equivalent for a $\V$-category $\A$: \begin{enumerate}\setlength\itemsep{0.25em}
		\item $\A$ is conically $\alpha$-accessible;
		\item $\A$ is the free cocompletion of a small category under $\alpha$-filtered colimits. 
	\end{enumerate}
\end{prop}
\begin{proof}
	This is a direct consequence of \cite[Proposition~4.3]{KS05:articolo}.
\end{proof}

\begin{obs}
	In the same spirit of Remark~\ref{phi-acc-funct}, it follows that a $\V$-functor $F\colon\A\to\B$ out of a conically $\alpha$-accessible $\V$-category $\A$ is conically $\alpha$-accessible if and only if it is the left Kan extension of its restriction to $\A^c_\alpha$.
\end{obs}

As we can raise the index of accessibility of an ordinary accessible category, we can do the same with conical accessibility:

\begin{cor}
	Given any $\alpha$-accessible $\V$-category $\A$ and a regular cardinal $\beta\rhd\alpha$, then $\A$ is conically $\beta$-accessible.
\end{cor}
\begin{proof}
	By Proposition~\ref{conicallyaccproperties} $\A_0$ is $\alpha$-accessible and $(\A_{\alpha}^c)_0=(\A_0)_{\alpha}$; moreover by Theorem~\ref{raising} $\A_0$ is $\beta$-accessible and $(\A_0)_{\beta}$ is given by the closure of $(\A_0)_{\alpha}$ in $\A$ under $\beta$-small $\alpha$-filtered colimits. It follows that $(\A_0)_{\beta}\subseteq (\A_{\beta}^c)_0$ (since the latter is closed under existing $\beta$-small colimits). As a consequence every element of $\A$ is a $\beta$-filtered colimit of objects from $\A_{\beta}^c$ (since that is true in $\A_0$). This is enough to imply that $\A$ is conically $\beta$-accessible.
\end{proof}

\begin{cor}\label{preserving-con-presentables}
	For any accessible $\V$-functor $F\colon\A\to\B$ between conically accessible $\V$-categories, there exists an $\alpha$ such that $F$ preserves the conically $\beta$-presentable objects for each $\beta\rhd\alpha$.
\end{cor}
\begin{proof}
	Direct consequence of Proposition~\ref{conicallyaccproperties}(4) and Corollary~\ref{preserving-presentables}.
\end{proof}

\vspace{6pt}

\subsection{Accessible vs.\ conically accessible}\label{accvscon}

The aim of this section is to compare the two notions of accessibility just introduced. In general, for a $\V$-category $\A$ with $\alpha$-flat colimits, we only have the inclusion $\A_\alpha\subseteq \A_\alpha^c$ (since every $\alpha$-filtered colimit is $\alpha$-flat). This inclusion is not an equality in general and moreover conically accessibility does not imply accessibility (since  some $\alpha$-flat colimits may not be $\alpha$-filtered, see \cite{LT21:articolo}). However for many significant base of enrichment the two notions do coincide, or differ only by Cauchy completeness \cite[Section~3-4]{LT21:articolo}. In the remainder of this section we give conditions on when a conically accessible $\V$-category is accessible, and prove that every $\alpha$-accessible $\V$-category is conically $\alpha^+$-accessible.

\begin{Def}
	We say that a $\V$-category $\A$ is {\em accessible} if it is $\alpha$-accessible for some $\alpha$; we say that it is {\em conically accessible} if it is conically $\alpha$-accessible for some $\alpha$. We say that a $\V$-functor $F\colon\A\to\K$ (not necessarily between accessible $\V$-categories) is {\em accessible} if $\A$ has and $F$ preserves $\alpha$-flat colimits for some $\alpha$; we say that it is {\em conically accessible} if $\A$ has and $F$ preserves $\alpha$-filtered colimits for some $\alpha$. If $F$ is fully faithful we say that $\A$ is respectively {\em accessibly embedded} and {\em conically accessibly embedded}.
\end{Def}

For this section we won't be considering $\Phi$-accessible $\V$-categories for a general weakly sound class $\Phi$, these will come into play again in Section~\ref{sketch}.

In the first part of this section we give conditions for conical accessibility to imply accessibility. As mentioned above, it's not true in general because some flat-weighted colimits might be missing in the $\V$-category in question; things change if the $\V$-category is complete or cocomplete:

\begin{prop}\label{cocompleteacc}
	Let $\A$ be a complete or $\alpha$-cocomplete $\V$-category; then $\A$ has $\alpha$-flat colimits if and only if it has $\alpha$-filtered colimits. A $\V$-functor from such an $\A$ preserves $\alpha$-flat colimits if and only if it preserves $\alpha$-filtered colimits. 
\end{prop}
\begin{proof}
	Assume first that $\A$ is $\alpha$-cocomplete and consider an $\alpha$-flat weight $M:\C^{op}\to\V$ together with a diagram $H:\C\to\A$. Let $J:\C\hookrightarrow\D$ be the inclusion of $\C$ into its free cocompletion under $\alpha$-small colimits. Since $\A$ has them we can consider $H':=\tx{Lan}_JH$, while on the weighted side we take $M':=\tx{Lan}_{J^{op}}M$. By Lemma \ref{flat-restriction} the weight $M'$ is still $\alpha$-flat and its domain is $\alpha$-complete. Therefore by Proposition~\ref{flat-char} we can write $M'\cong\tx{colim}YF$ as an $\alpha$-filtered colimit of representables; here $Y:\D\to[\D^{op},\V]$ is the Yoneda embedding and $F:\E_{\V}\to\D$ is a functor with $\alpha$-filtered domain. As a consequence we obtain (each side existing if the other does):
	\begin{equation*}
		\begin{split}
			M*H&\cong M*H'J\cong M'*H'\\
			&\cong (\tx{colim}YF)* H'\\
			&\cong \tx{colim} (YF*H')\\
			&\cong \tx{colim} (H'F)\\
		\end{split}
	\end{equation*}
	Thus the existence and preservation of the $\alpha$-flat colimit $M*H$ is equivalent to that of the $\alpha$-filtered colimit $\tx{colim} H'F$.
	
	The case when $\A$ is complete goes exactly as above with the only difference that we consider $H':=\tx{Ran}_JH$ (instead of $\tx{Lan}_JH$): this exists because $\A$ is assumed to be complete. The other arguments apply identically since the isomorphism $ M*H\cong M'*H'$ still holds.
\end{proof}

A direct consequence is the following:

\begin{cor}
	The following are equivalent for a complete or $\alpha$-cocomplete $\V$-category $\A$:\begin{enumerate}\setlength\itemsep{0.25em}
		\item $\A$ is $\alpha$-accessible;
		\item $\A$ is conically $\alpha$-accessible;
	\end{enumerate} 
	In this case $\A_\alpha=\A_\alpha^c$ and $\A$ is locally $\alpha$-presentable in the sense of \cite{Kel82:articolo}.
\end{cor}

Let's now see how the accessibility of the underlying category relates to that of the $\V$-category.

\begin{prop}\label{conical-flat}
	Let $\A$ be a $\V$-category for which $\A_0$ is accessible; then: \begin{enumerate}\setlength\itemsep{0.25em}
		\item  $\A$ is conically accessible if and only if it has conical $\alpha$-filtered colimits for some $\alpha$ and every object is conically presentable;
		\item $\A$ is accessible if and only if it has $\alpha$-flat colimits for some $\alpha$ and every object is presentable.
	\end{enumerate}
\end{prop}
\begin{proof}
	$(1)$. If $\A$ is conically $\alpha$-accessible then it has conical $\alpha$-filtered colimits by definition; moreover every object is conically presentable being a small colimit of $\alpha$-presentable objects.\\
	Conversely, let $\beta\geq\alpha$ be such that $\A_0$ is $\beta$-accessible and let $\gamma\rhd\beta$ be such that $\A$ has conical $\gamma$-filtered colimits and
	$$ (\A_0)_\beta\subseteq (\A_\gamma^c)_0; $$
	this exists since $(\A_0)_\beta$ is small and each object of $\A$ is conically presentable. It follows that
	$$ (\A_\gamma^c)_0=(\A_0)_\gamma. $$
	Indeed, the inclusion $(\A_\gamma^c)_0\subseteq(\A_0)_\gamma$ is always true (since the unit $I$ of $\V$ is $\gamma$-presentable); on the other hand every $X\in(\A_0)_\gamma$ is a $\gamma$-small $\beta$-filtered colimit of objects from $(\A_0)_\beta$; since $\A$ has $\gamma$-filtered colimits, the colimit expressing $X$ is actually enriched; hence $X$ is a $\gamma$-small colimit of conically $\beta$-presentable objects in $\A$, and this makes $X$ a conically $\gamma$-presentable object of $\A$.\\
	Given the equality above, and the fact that $\A_0$ is $\gamma$-accessible, it follows that $\A_\gamma^c$ generates $\A$ under $\gamma$-filtered colimits; therefore $\A$ is conically accessible.
	
	$(2)$. The proof is essentially the same as above, one just has to replace conical presentability with actual presentability.  
\end{proof}

\begin{cor}\label{conically-acc}
	Let $\K$ be conically accessible, $J\colon\A\hookrightarrow\K$ be conically accessibly embedded, and $\A_0$ be accessible; then $\A$ is conically accessible. If moreover $\K$ is accessible and $\A$ is accessibly embedded, then it is accessible.
\end{cor}
\begin{proof}
	This is a direct consequence of the previous Proposition since each object of $\A$ will be conically presentable in $(1)$ and presentable in $(2)$.
\end{proof}

We now turn to proving that every accessible $\V$-category is also conically accessible; the next Lemma will be an important step.

\begin{lema}\label{intersection-acc}
	Let $\K$ be a locally $\alpha$-presentable $\V$-category, and $\A$ and $\B$ be two conically $\alpha$-accessible full subcategories of $\K$ for which the inclusions preserve the conically $\alpha$-presentable objects. Let $\C:=\A\cap\B$ and $\beta\rhd\alpha$:\begin{enumerate}\setlength\itemsep{0.25em}
		\item if $\A$ and $\B$ are closed under $\alpha$-filtered colimits in $\K$, then $\C$ is conically $\beta$-accessible and closed under $\beta$-filtered colimits in $\K$.
		\item if $\A$ and $\B$ are closed under $\alpha$-flat colimits in $\K$, then  $\C$ is $\beta$-accessible and closed under $\beta$-flat colimits in $\K$; moreover $\C_\beta=\C^c_\beta$.
	\end{enumerate} 
\end{lema}
\begin{proof}
	$(1)$. Consider $\beta\rhd\alpha$; then $\A$ and $\B$ are conically $\beta$-accessible, the inclusions in $\K$ still preserve $\beta$-filtered colimits and conically $\beta$-presentable objects (since these are $\beta$-small $\alpha$-filtered colimits of the conically $\alpha$-presentable ones). Note first that $\C$ is closed in $\A,\B$, and $\K$ under $\beta$-filtered colimits since both $\A$ and $\B$ are so in $\K$. 
	
	Let $\C'\subseteq \C$ be the intersection $\C'=\A^c_\beta\cap\B^c_\beta$; then $\C'\subseteq\C^c_\beta$ and to prove $(1)$ it is enough to show that $\C'$ generates $\C$ under $\beta$-filtered colimits. For any $X\in\C$ consider the slice $\C'_0/X$ and the inclusion $J\colon\C'_0/X\to(\A^c_\beta)_0/X$; we wish to prove that $J$ is final. This will suffice since then $\C'_0/X$ will be $\beta$-filtered (because $(\A^c_\beta)_0/X$ is and the inclusion is fully faithful) and the colimit of $\pi_X\colon\C'_0/X\to \C$ will be $X$ (because the colimit of $(\A^c_\beta)_0/X\to \A$ is $X$ and $\C$ is closed in $\A$ under $\beta$-filtered colimits).
	
	So we are reduced to proving that $J\colon\C'_0/X\to(\A^c_\beta)_0/X$ is final; which is as saying that every map $A\to X$ with $A\in\A^c_\beta$ factors through some $C\in\C'$ (the fact that any two such factorizations are connected will follow from this plus the filteredness of $(\A^c_\beta)_0/X$ and fully faithfulness of the inclusion).
	Fix then a map $f\colon A\to X$ with $A\in\A^c_\beta$, we regard this as a morphism in $\K$ and construct a $\beta$-small chain $(d_{i,j}\colon K_i\to K_j)_{i<j<\alpha}$ of conically $\beta$-presentable objects in $\K$ together with a cocone $(c_i\colon K_i\to X)_{i<\alpha}$. Set $K_0=A$ and $c_0=f$, then we alternate elements of $\A^c_\beta$ and $\B^c_\beta$ as follows (taking colimits in $\K$ at the limit steps). Assume to have $K_i$ and $c_i$ for $i=\lambda +2n$ with $\lambda$ limit; then $c_i\colon K_i\to X$ factors through some $B\in\B^c_\beta$ since $K_i$ is conically $\beta$-presentable (remember that $\A^c_\beta,\B^c_\beta\subseteq \K^c_\beta$) and $X$, being in $\B$, is a $\beta$-filtered colimit of objects from $\B^c_\beta$. Let then $K_{i+1}=B$ with $d_{i,i+1}$ and $c_{i+1}$ given by the factorization. If $i=\lambda +2n+1$ we argue as above but inverting the roles of $\A$ and $\B$. Finally if $i=\lambda$ is limit, we take $K_i$ to be the colimit of the chain $(K_j)_{j<i}$ in $\K$ and consider the induced factorizations. Let $C:=\tx{colim}_{i<\alpha}K_i$ in $\K$; then by construction we have a factorization of $f$ through $C$. Moreover the sub-chains of $(K_i)_{i<\alpha}$ spanned by the objects in $\A^c_\beta$ and $\B^c_\beta$ are final; thus $C$ is both in $\A$ and in $\B$ and hence in $\C$. Finally, since the chain involved was $\beta$-small, $C$ is actually an object of $\A^c_\beta\cap\B^c_\beta=\C'$ as required. Note that, since in $\C'$ idempotents split, this also implies that $\C^c_\beta=\A^c_\beta\cap\B^c_\beta$.
	
	$(2)$. Since $\A$ and $\B$ are closed in $\K$ under $\beta$-flat colimits, also $\C$ is. Moreover, thanks to point $(1)$, to prove that $\C$ is $\beta$-accessible it's enough to show that $\C^c_\beta\subseteq\C_\beta$. Let $X\in\C^c_\beta$ and denote by $J$ the inclusion of $\C$ in $\K$; since $X$, seen as an object of $\A$, is conically $\beta$-presentable and the inclusion of $\A$ in $\K$ preserves conically $\beta$-presentable objects, it follows that $JX\in\K^c_\beta$. But $\K$ is locally $\beta$-presentable and thus $\K^c_\beta=\K_\beta$; as a consequence $\C(X,-)\cong\K(JX,J-)$ preserves $\beta$-flat colimits and hence $X\in\C_\beta$. This proves that $\C^c_\beta\subseteq\C_\beta$; since the other inclusion always holds, it follows that $\C^c_\beta=\C^\beta$.
\end{proof}

As promised, the next result says that we can raise the index of accessibility of an accessible $\V$-category and at the same time make the $\V$-category conically accessible. This can be seen as a sharpening of \cite[Theorem~7.10]{BQR98}; in fact, in \cite{BQR98} the choice of $\beta$ depends on the $\V$-category $\A$ taken into consideration, while in our result it depends only on $\alpha$ and the {\em sharply less} relation.

\begin{teo}\label{flat-to-conical}
	If $\A$ is an $\alpha$-accessible $\V$-category, then for any $\beta\rhd\alpha$ the following hold: \begin{enumerate}\setlength\itemsep{0.25em}
		\item $\A$ is $\beta$-accessible;
		\item $\A$ is conically $\beta$-accessible;
		\item $\A_0$ is an ordinary $\beta$-accessible category;
		\item $(\A_{\beta})_0=(\A_\beta^c)_0=(\A_0)_\beta$.
	\end{enumerate}
\end{teo}
\begin{proof}
	Let $\C=\A_\alpha^{op}$; since $\A$ is $\alpha$-accessible we can write it as $\A\simeq\alpha\tx{-}\tx{Flat}(\C,\V)$, with inclusion $J\colon\A\hookrightarrow[\C,\V]$. Now consider $H\colon\C\hookrightarrow\D$ to be the free completion of $\C$ under $\alpha$-small limits; we show first that $\A$ can be identified with the intersection
	\begin{center}
		\begin{tikzpicture}[baseline=(current  bounding  box.south), scale=2]
			
			\node (a0) at (0,0.9) {$\A$};
			\node (b0) at (1.2,0.9) {$[\C,\V]$};
			\node (c0) at (0,0) {$[\C,\V]$};
			\node (d0) at (1.2,0) {$[\D,\V]$};
			\node (e0) at (0.2,0.7) {$\lrcorner$};
			
			\path[font=\scriptsize]
			
			(a0) edge [right hook->] node [above] {$J$} (b0)
			(a0) edge [right hook->] node [left] {$J$} (c0)
			(b0) edge [right hook->] node [right] {$\tx{Lan}_H$} (d0)
			(c0) edge [right hook->] node [below] {$\tx{Ran}_H$} (d0);
		\end{tikzpicture}	
	\end{center}
	where we are embedding $[\C,\V]$ in $[\D,\V]$ in two different ways. To prove that, it's enough to show that a $\V$-functor $F\colon\C\to\V$ is $\alpha$-flat if and only if $\tx{Lan}_HF\cong \tx{Ran}_HF$. If $F$ is $\alpha$-flat then $\tx{Lan}_HF$ is $\alpha$-flat as well by Lemma~\ref{flat-restriction} and therefore is $\alpha$-continuous; since $\D$ is the free completion of $\C$ under $\alpha$-small limits and $(\tx{Lan}_HF)J\cong F$ this means exactly that $\tx{Lan}_HF\cong \tx{Ran}_HF$. Vice versa, if $\tx{Lan}_HF\cong \tx{Ran}_HF$ then $\tx{Lan}_HF$ is $\alpha$-continuous and hence $\alpha$-flat; thus $F$ is $\alpha$-flat itself again by Lemma~\ref{flat-restriction}.
	
	To conclude the proof of $(1),(2),$ and $(3)$ it's now enough to show that, for this intersection, the hypotheses of Lemma~\ref{intersection-acc} are satisfied. The $\V$-categories $[\C,\V]$ and $[\D,\V]$ are locally $\alpha$-presentable (for any $\alpha$) and hence conically $\alpha$-accessible; moreover their $\alpha$-presentable and conically $\alpha$-presentable objects coincide. Now $\tx{Lan}_H\colon [\C,\V]\hookrightarrow[\D,\V]$ is cocontinuous and sends representable functors to representables; therefore it preserves all $\alpha$-flat colimits and the $\alpha$-presentable objects (since these coincide with the $\alpha$-small colimits of representables). It remains to consider $\tx{Ran}_H\colon [\C,\V]\hookrightarrow[\D,\V]$; this identifies $[\C,\V]$ with the full subcategory of $[\D,\V]$ spanned by the $\alpha$-continuous functors. Since these are closed under $\alpha$-flat colimits it follows at once that $\tx{Ran}_H$ preserves $\alpha$-flat colimits and we are only left to prove that it preserves the $\alpha$-presentable objects. Under the identification just described, the $\alpha$-presentable objects of $[\C,\V]$ correspond to the representables in $\alpha\tx{-}\tx{Cont}(\D,\V)$ by the enriched Gabriel-Ulmer duality; therefore $\tx{Ran}_H$ sends the $\alpha$-presentable objects to the representables in $[\D,\V]$, and these are $\alpha$-presentable. In conclusion, we can apply Lemma~\ref{intersection-acc} to obtain $(1),(2),$ and $(3)$. Point $(4)$ is now a consequence of Proposition~\ref{conicallyaccproperties}.
\end{proof}

\begin{obs}
	The Theorem above implies in particular that if $\A$ is $\alpha$-accessible then it is also conically $\alpha^+$-accessible and (conically) $\beta$-accessible for arbitrarily large regular cardinals $\beta$. We don't know yet whether every $\alpha$-accessible $\V$-category is also conically $\alpha$-accessible or not.
\end{obs}

\begin{cor}\label{preserving-flat-presentables}
	Given an accessible $\V$-functor $F\colon \A\to\B$ between accessible $\V$-categories, there exists $\alpha$ such that $F$ preserves the $\beta$-presentable objects for each $\beta\rhd\alpha$.
\end{cor}
\begin{proof}
	Direct consequence of Corollary~\ref{flat-to-conical} above and Corollary~\ref{preserving-presentables}.
\end{proof}

\begin{prop}\label{accessible-functors}
	A $\V$-functor $F\colon \A\to\B$ between accessible $\V$-categories is accessible if and only if it is conically accessible. 
\end{prop}
\begin{proof}
	Every accessible functor is conically accessible. Conversely let $F\colon \A\to\B$ be conically $\alpha$-accessible for some $\alpha$ and consider $\beta\rhd\alpha$, by Theorem~\ref{flat-to-conical} it follows that $\A$ is conically $\beta$-accessible. Since $F$ is also conically $\beta$-accessible, it is the left Kan extension of its restriction to $\A_{\beta}^c=\A_{\beta}$, which is made of $\beta$-presentable objects; thus $F$ preserves $\beta$-flat colimits as well.
\end{proof}

\section{The main results}\label{mainresults-acc}

The aim of this section is to introduce and work with the virtual notions discussed in the introduction. We prove the main results in Section~\ref{virt-reflective}, \ref{vir-orthogonal}, and \ref{virtual-main}, and then compare these with those already known in the literature in Section~\ref{cre+p}. In the first subsection we establish once more the connection between accessible $\V$-categories and sketches.

\subsection{Accessibility and sketches}\label{sketch}

The relationship between accessible $\V$-categories and sketches already appeared in \cite{BQR98}; however their proof relies on some (non trivial) results on ordinary accessible categories; here we give a proof that is only based on Section~\ref{sectionacc} and on some standard results about locally presentable $\V$-categories.

First we need to recall the notion of sketch, which in a general enriched context was already considered in \cite{Kel82:libro}:

\begin{Def}
	A {\em sketch} is the data of a triple $\S=(\B,\mathbb{L},\mathbb{C})$ where: \begin{itemize}\setlength\itemsep{0.25em}
		\item $\B$ is a small $\V$-category;
		\item $\mathbb{L}$ is a set of cylinders in $\B$: $\V$-natural transformations $c\colon  N \to\B(B,H-)$, where $ N \colon \D\to\V$ is a weight, $H\colon \D\to\B$ is a $\V$-functor, and $B$ is an object of $\B$;
		\item $\mathbb{C}$ is a set of cocylinders in $\B$: $\V$-natural transformations $d\colon  M \to\B(K-,C)$, where $ M \colon \E^{op}\to\V$ is a weight, $K\colon \E\to\B$ is a $\V$-functor, and $C$ is an object of $\B$.
	\end{itemize}
	A sketch for which $\mathbb{C}$ is empty is called a {\em limit sketch}.
\end{Def}

\begin{Def}
	A model of a sketch $\S=(\B,\mathbb{L},\mathbb{C})$ is a $\V$-functor $F\colon \B\to\V$ which transforms each cylinder of $\mathbb{L}$ into a limit cylinder in $\V$, and each cocylinder of $\mathbb{C}$ into a colimit cocylinder in $\V$. We denote by $\tx{Mod}(\S)$ the full subcategory of $[\B,\V]$ spanned by the models of $\S$.
\end{Def}

The $\V$-categories of models of limit sketches characterize locally presentable $\V$-categories (see \cite[Section~10]{Kel82:articolo} or \cite[Corollary~7.4]{BQR98}).

\begin{prop}\label{phi-limitcolimit}
	Any $\Phi$-accessible $\V$-category $\A$ is equivalent to the $\V$-category of models of a sketch involving colimits and $\Phi$-limits. 
\end{prop}
\begin{proof}
	Let $\A$ be $\Phi$-accessible; then by Proposition~\ref{acc=flat} we can write $\A\simeq\Phi\tx{-}\tx{Flat}(\C,\V)$ for some small $\C$; thus it's enough to prove that $\Phi\tx{-}\tx{Flat}(\C,\V)$ is the $\V$-category of models of a suitable sketch. Let $\D$ be the closure of $\C$ in $[\C^{op},\V]$ under $\Phi$-limits; then by left Kan extending along the inclusion of $\C$ in $\D$, the $\V$-category $\Phi\tx{-}\tx{Flat}(\C,\V)$ becomes equivalent to the full subcategory of $[\D,\V]$ spanned by those $\Phi$-continuous $\V$-functors $F\colon \D\to\V$ which preserve some specified weighted colimits (those that exhibit each object of $\D$, seen in $[\C^{op},\V]$, as a weighted colimit of representables). This is clearly the $\V$-category of models of a sketch on $\D$ involving $\Phi$-limits and colimits.
\end{proof}

The converse doesn't hold even when $\Phi$ is the class of $\alpha$-small limits, see for instance \cite[Remark~2.59]{AR94:libro}.

In the proof of the theorem below we'll make use of the enriched notion of orthogonality class, which can be found for example in \cite[Section~6.2]{Kel82:libro}. 

The equivalence between $(1)$ and $(3)$ below already appeared in \cite[Theorems~7.6/7.8]{BQR98}.

\begin{teo}\label{sketches-accessible}
	Let $\A$ be a $\V$-category; the following are equivalent:\begin{enumerate}
		\item $\A$ is accessible (that is, $\alpha$-accessible for some $\alpha$);
		\item $\A$ is $\Phi$-accessible for some small weakly sound class $\Phi$;
		\item $\A$ is sketchable.	
	\end{enumerate}
\end{teo}
\begin{proof}
	The implication $(1)\Rightarrow(2)$ is trivial and $(2)\Rightarrow(3)$ is given by Proposition~\ref{phi-limitcolimit} above.
	
	$(3)\Rightarrow(1)$. Let $\A=\tx{Mod}(\S)$ be the $\V$-category of models of a sketch $\S=(\B,\mathbb{L},\mathbb{C})$; then $\A=\tx{Mod}(\B,\mathbb{L}) \cap \tx{Mod}(\B,\mathbb{C})$ can be seen as the intersection of the limit part and the colimit part. Let us focus on $\tx{Mod}(\B,\mathbb{C})$ as a full subcategory of $[\B,\V]$. Let $\C$ be the $\V$-category obtained from $\B$ freely adding the colimits of each diagram $( M ,H)$ appearing in $\mathbb{C}$, denote by $J\colon \B\hookrightarrow\C$ the inclusion. Then $\tx{Lan}_J\colon [\B,\V]\to[\C,\V]$ is fully faithful and, for any $F\colon \B\to\V$, its left Kan extension $\tx{Lan}_JF$ preserves the specified colimits. Consider now the cylinders $c\colon  M \to\B(H-,B)$ appearing in $\mathbb{C}$, by construction these correspond to maps $\bar{c}\colon  M *JH\to JB$ in $\C$. Denote by $\M$ the family of morphisms in $[\C,\V]$ given by $\C(\bar{c},-)$ for each $c\in\mathbb{C}$; then the square below is a pullback.
	\begin{center}
		\begin{tikzpicture}[baseline=(current  bounding  box.south), scale=2]
			
			\node (a0) at (0,0.9) {$\tx{Mod}(\B,\mathbb{C})$};
			\node (b0) at (1.2,0.9) {$[\B,\V]$};
			\node (c0) at (0,0) {$\M^\perp$};
			\node (d0) at (1.2,0) {$[\C,\V]$};
			\node (e0) at (0.25,0.65) {$\lrcorner$};
			
			\path[font=\scriptsize]
			
			(a0) edge [right hook->] node [above] {} (b0)
			(a0) edge [right hook->] node [left] {} (c0)
			(b0) edge [right hook->] node [right] {$\tx{Lan}_J$} (d0)
			(c0) edge [right hook->] node [below] {} (d0);
		\end{tikzpicture}	
	\end{center}
	Indeed, $F\colon \B\to\V$ is a model of $(\B,\mathbb{C})$ if and only if for each $c\in\mathbb{C}$ as above the induced map $\tilde{c}\colon  M *FH\to FB$ is an isomorphism; but $FB\cong (\tx{Lan}_JF)JB$, while $ M *FH\cong  M *(\tx{Lan}_JF)JH\cong \tx{Lan}_JF( M *JH)$, and $\tilde{c}$ corresponds under these isomorphisms to $\tx{Lan}_JF(\bar{c})$. It follows that $F$ is a model of the colimit sketch if and only if $\tx{Lan}_JF(\bar{c})$ is an isomorphism for each $c$, which is equivalent to $\tx{Lan}_JF$ being orthogonal to $\C(\bar{c},-)$ for each $c$.
	
	In conclusion we can express the $\V$-category $\A$ as the intersection in $[\C,\V]$ of the accessibly embedded subcategories $\M^\perp$ and $\tx{Mod}(\B,\mathbb{L})$. Now it's enough to observe that $\tx{Mod}(\B,\mathbb{L})$ and $\M^\perp$ are locally presentable; the first being the $\V$-category of models of a limit sketch, and the latter being an accessibly embedded and reflective subcategory of $[\C,\V]$ by \cite[Theorem~6.5]{Kel82:libro}. As a consequence $\A$ is accessible by Corollary~\ref{preserving-con-presentables} and Lemma~\ref{intersection-acc}.
\end{proof}

\vspace{6pt}

\subsection{Accessibility and the free completion}

Here we collect a few results about the free completion under small limits of an accessible $\V$-category; this will serve as an introduction to the virtual concepts we consider in Section~\ref{virt-reflective}.

Given a $\V$-category $\A$ we denote by $\P^\dagger\A$ its free completion under (small) limits; this can be seen as the full subcategory of $[\A,\V]^{op}$ spanned by the small limits of representables. More common is the free cocompletion under colimits, denoted by $\P\A$; this has been studied in \cite{DL07} and is related to the free completion under limits through the duality $\P^\dagger\A=\P(\A^{op})^{op}$. Note moreover that, when $\C$ is a small $\V$-category,  $\P^\dagger\C=[\C,\V]^{op}$. 

\begin{Def}
	Given a $\V$-category $\A$, we say that a small full subcategory $H\colon\G\hookrightarrow\A$ is {\em closed under virtual $\Phi$-colimits} in $\A$ if $\P^\dagger H\colon\P^\dagger\G\hookrightarrow\P^\dagger\A$ is $\Phi$-cocontinuous.
\end{Def}

Notice that, since $\G$ is small, $\P^\dagger \G=[\G,\V]^{op}$ is cocomplete and therefore $\P^\dagger H$ will be a genuine $\Phi$-cocontinuous $\V$-functor. The fact $\P^\dagger\A$ has enough colimits is not needed to prove the result below; however that will be a consequence of Proposition~\ref{accvirtcoco} where we prove that $\P^\dagger\A$ is cocomplete whenever $\A$ is accessible.

Recall that $H\colon\G\hookrightarrow\A$ is a strong generator if the $\V$-functor $\A(H,1)\colon\A\to[\G^{op},\V]$ is conservative. The following is an equivalent way of characterizing $\Phi$-accessible $\V$-categories. 

\begin{prop}\label{acc-strong}
	Let $\Phi$ be a small weakly sound class and $\A$ be a $\V$-category with $\Phi$-flat colimits. The following are equivalent for $\G\subseteq\A_\Phi$:
	\begin{enumerate}
		\item $\G$ exhibits $\A$ as a $\Phi$-accessible $\V$-category;
		\item $\G$ is a small strong generator of $\A$ that is closed under virtual $\Phi$-colimits.
	\end{enumerate}
\end{prop}
\begin{proof}
	Let $Y\colon \G^{op}\to[\G,\V]=\P^\dagger (\G)^{op}$ be the Yoneda embedding and $A\in\A$; then 
	$$\tx{Lan}_Y \A(H-,A)\cong\tx{ev}_A\circ (\P^\dagger H)^{op}$$ 
	for any $\A$ in $\A$, indeed this follows from the fact that the $\V$-functor on the right-hand-side is cocontinuous and restricts to $\A(H-,A)$.
	
	Assume now that $\A$ is $\Phi$-accessible and let $\G=\A_\Phi$; the $\V$-functor $\A(H-,A)$ is $\Phi$-flat by Proposition~\ref{A_alpha}; therefore $\tx{Lan}_Y \A(H-,A)$ preserves all $\Phi$-limits. By the isomorphism above it follows that $\tx{ev}^{op}_A\circ \P^\dagger H$ is $\Phi$-cocontinuous; thus, since $\Phi$-colimits in $\P^\dagger \A$ (when they exist) are computed pointwise, $\P^\dagger H$ is $\Phi$-cocontinuous too. It follows that $\G$ is a small strong generator (being dense) and is closed in $\A$ under virtual $\Phi$-colimits.
	
	Conversely, assume that there exists $H\colon\G\hookrightarrow\A$ as in $(2)$. Consider the $\V$-functor $W\colon\A\to[\G^{op},\V]$ defined by $W=\A(H,1)$; since $\G$ is a strong generator made of $\Phi$-presentable objects it follows that $W$ is conservative and preserves $\Phi$-flat colimits. Moreover, since $\P^\dagger H$ is $\Phi$-cocontinuous and $\tx{ev}_A$ is continuous for any $A\in\A$, it follows that $\tx{Lan}_Y \A(H-,A)$ is $\Phi$-continuous, and hence $WA=\A(H-,A)$ is $\Phi$-flat. As a consequence, given any $A\in\A$, we have that 
	$$WA\cong WA*Y\cong WA*WH\cong W(WA*H);$$
	therefore $A\cong WA*H$ (by conservativeness of $W$) is a $\Phi$-flat colimit of elements of $\G$, showing that $\A$ is $\Phi$-accessible
\end{proof}

Thanks to this, one obtains easily the standard characterization theorem of locally presentable $\V$-categories:

\begin{cor}
	Let $\Phi$ be a small weakly sound class. The following are equivalent for a cocomplete $\V$-category $\A$:\begin{enumerate}
		\item $\A$ is $\Phi$-accessible;
		\item $\A$ has a small strong generator made of $\Phi$-presentable objects. 
	\end{enumerate}
\end{cor}
\begin{proof}
	The implication $(1)\Rightarrow(2)$ is trivial. For $(2)\Rightarrow(1)$ consider a small strong generator $\G\subseteq\A_\Phi$ and take its closure $\G'$ in $\A$ under $\Phi$-colimits. Then $\G'$ is still strongly generating and is also closed in $\A$ under virtual $\Phi$-colimits by the dual of \cite[Remark~6.6]{DL07}. It follows by the proposition above that $\A$ is $\Phi$-accessible.
\end{proof}

For a locally $\alpha$-presentable $\V$-category $\K$ we know that the $\alpha$-presentable objects are closed in $\K$ under $\alpha$-small colimits. This can't be said in a general $\alpha$-accessible category (since those colimits may not exist); however, by taking $\G=\A_\Phi$ in Proposition~\ref{acc-strong} above, we obtain:

\begin{prop}
	Let $\Phi$ be a small weakly sound class of weights. For any $\Phi$-accessible $\V$-category $\A$ the full subcategory $\A_{\Phi}$ is closed in $\A$ under virtual $\Phi$-colimits. The same holds for a conical $\alpha$-accessible $\A$ with $\A_{\alpha}^c$ in place of $\A_{\Phi}$.
\end{prop}
\begin{proof}
	For the first part simply take $\G=\A_\Phi$ in the proposition above. For the latter apply the same proof by noticing that $\A(H-,A)$ will still be $\alpha$-flat even when $H$ is the inclusion of $\A^c_\alpha$ in $\A$.
\end{proof}

If we see $\P^\dagger\A$ as a full subcategory of $[\A,\V]^{op}$ then we have a nice way of describing its elements in the case where $\A$ is accessible:

\begin{prop}\label{acc=small}
	For any accessible $\V$-category $\A$ the free completion $\P^\dagger \A$ consists exactly of the accessible functors from $\A$ to $\V$. For any conically accessible $\V$-category $\A$ the free completion $\P^\dagger \A$ consists exactly of the conically accessible functors from $\A$ to $\V$.
\end{prop}
\begin{proof}
	Let $F\colon\A\to\V$ be an object of $\P^\dagger \A$; then $F=\tx{Lan}_HFH$ for some small $H\colon \C\hookrightarrow\A$. Since $\C$ is small, we can now consider $\alpha$ for which $\C\subseteq\A_{\alpha}$; it follows that $F$ is also the left Kan extension of its restriction to $\A_{\alpha}$, and hence it preserves $\alpha$-flat colimits. Conversely, every $F\colon\A\to\V$ which preserves $\alpha$-flat colimits (for an $\alpha$ for which $\A$ is $\alpha$-accessible) is the left Kan extension of its restriction to $\A_{\alpha}$; hence it is a small presheaf. The same argument applies to the conical case.
\end{proof}

\begin{cor}
	For any accessible $\V$-category $\A$ there is an adjunction:
	\begin{center}
		\begin{tikzpicture}[baseline=(current  bounding  box.south), scale=2]
			
			\node (a) at (0,0.8) {$\P^\dagger (\A_0)$};
			\node (b) at (1.4,0.8) {$(\P^\dagger \A)_0$};
			\node (e) at (0.7,0.9) {$\perp$};
			
			\path[font=\scriptsize]
			
			(a) edge [bend left, <-] node [above] {$L$} (b)
			(a) edge [->] node [below] {$R$} (b);
			
		\end{tikzpicture}	
	\end{center} 
	where $R$ is the unique continuous functor induced by the universal property of $\P^\dagger (\A_0)$ applied to the underlying functor of the inclusion $Z\colon \A\to\P^\dagger \A$, and $L$ is given pointwise by $LF=\V_0(I,F_0-)$. In fact it suffices that $\A$ be conically accessible.
\end{cor}
\begin{proof}
	Let $\bar{Z}\colon \A_0\hookrightarrow\P^\dagger (\A_0)$ be the inclusion of the free completion of $\A_0$, so that $R\bar{Z}\cong Z_0$. Then $R$ has a left adjoint $L$ if and only if for any $G$ in $(\P^\dagger \A)_0$ there exists some $LG\in \P^\dagger (\A_0)$ such that
	$\P^\dagger (\A_0)(LG,-)\cong (\P^\dagger \A)_0(G,R-).$
	By restricting the isomorphism to $\A_0$ through $\bar{Z}$, this says that $LG\cong (\P^\dagger \A)_0(G,Z_0-)$; in particular it follows that a left adjoint exists if and only if $(\P^\dagger \A)_0(G,Z_0-)$ is a small functor for any $G$ in $(\P^\dagger \A)_0$. \\
	To conclude it's then enough to notice that $(\P^\dagger \A)_0(G,Z_0-)\cong \V_0(I,G_0-)$ is small by Proposition~\ref{acc=small} since both $\A$ and $\A_0$ are accessible.
\end{proof}

\vspace{6pt}

\subsection{Virtual left adjoints and virtual colimits}\label{virt-reflective}

In this section we introduce the {\em virtual} concepts that will help us in finding new characterizations of accessible $\V$-categories. The word virtual here refers to something that ``lives'' in the free completion $\P^\dagger\A$ of a $\V$-category $\A$. The relation between these concepts and those used in \cite{AR94:libro} for ordinary accessible categories will be discussed in Section~\ref{cre+p}.

Recall that a $\V$-functor $F\colon\A\to\K$ has a left adjoint if and only if $\K(X,F-)$ is representable for each $X\in\K$; generalizing this we make the following definition:

\begin{Def}\label{virtualleft}
	We say that a $\V$-functor $F\colon \A\to\K$ has a {\em virtual left adjoint} if for each $X\in\K$ the $\V$-functor $\K(X,F-)$ is small. If $F$ is moreover fully faithful we say that $\A$ is {\em virtually reflective} in $\K$.
\end{Def}

From now on, given a $\V$-category $\A$ we denote by $Z_\A\colon\A\hookrightarrow\P^\dagger \A$ the inclusion into its free completion; we'll drop the subscript $(-)_\A$ whenever it is safe to do so.

\begin{obs}
	Virtual left adjoints have been considered in the literature for different purposes. For instance, in \cite[3.4]{tholen1984pro} a functor is a {\em right $\mt D$-pro adjoint} (where $\mt D$ is the class of all small categories) if and only if it has a virtual left adjoint (in our sense). While, in the context of KZ-doctrines, virtual left adjoints were considered as the dual of \cite[Definition~1.1]{bunge1999bicomma} for the KZ-doctrine given by freely adding all small colimits.
\end{obs}

\begin{obs}
	When $\V=\bo{Set}$, Guitart and Lair consider in \cite[Section~5]{GL81:articolo} the notion of ``small locally free diagram'' ({\em petit diagramme localement libre}). Given a fully faithful functor $J\colon\A\to\K$, they say that an object $X\in\K$ has a small locally free diagram over $\A$ if there exists a diagram $H\colon\C\to\A$ for which:
	$$ \K(X,JA)\cong\colim\A(H-,A) $$
	naturally in $A\in\A$. Now, since $\colim\A(H-,A)\cong\P^\dagger\A(\lim ZH,ZA)\cong (\lim ZH)(A)$, this condition is simply saying that $\K(X,J-)$ is small. It follows that every element of $\K$ has a small locally free diagram over $\A$ if and only if $J$ has a virtual left adjoint. In Theorem~1 of the same paper they show that, given a sketch $\S$ on $\C$, every functor $F\colon\C\to\bo{Set}$ has a small locally free diagram over $\tx{Mod}(\S)$ (see also \cite[Theorem~2.3]{Ger94:PhD}); in our terminology this says that $\tx{Mod}(\S)$ is virtually reflective in $[\C,\bo{Set}]$. However they do not prove that every virtually reflective and accessibly embedded subcategory of $[\C,\bo{Set}]$ is accessible (or equivalently, sketchable); we do that in Proposition~\ref{vr-acc} for the more general enriched context.
\end{obs}

Other ways to recognize when a $\V$-functor has a virtual left adjoint are given below. 

\begin{prop}
	Let $F\colon\A\to\K$ be a $\V$-functor. The following are equivalent: \begin{enumerate}\setlength\itemsep{0.25em}
		\item $F$ has a virtual left adjoint;
		\item $F$ has a relative left adjoint with respect to the inclusion $Z_\A\colon\A\hookrightarrow\P^\dagger \A$;
		\item the induced continuous $\V$-functor $\P^\dagger F\colon \P^\dagger \A\to\P^\dagger \K$ has a left adjoint.
	\end{enumerate}
	Moreover the left adjoint $L\colon\P^\dagger \K\to\P^\dagger \A$ in $(3)$ is given by precomposition with $F$.
\end{prop}
\begin{proof}
	$(3)\Rightarrow (2)$. Let $L\colon\P^\dagger \K\to\P^\dagger \A$ be the left adjoint to $\P^\dagger F$ and let $X\in\P^\dagger\K$; then
	$$ LX(A)\cong \P^\dagger\A(LX,Z_\A A)\cong \P^\dagger\K(X,\P^\dagger FZ_\A A)\cong \P^\dagger\K(X,Z_\K FA)\cong X(FA). $$
	This proves that $L$, when it exists, is given by precomposition with $F$. Now the relative left adjoint to $F$ is given by the composite $LZ_\K\colon \K\to\P^\dagger\A$.
	
	$(2)\Rightarrow(3)$. Let $L'\colon\K\to\P^\dagger\A$ be a relative left adjoint to $F$ with respect to $Z_\A$. Then $L\colon =\tx{Ran}_{Z_\K}L'\colon\P^\dagger\K\to\P^\dagger\A$ exists (since $\P^\dagger\A$ is complete) and is a left adjoint to $\P^\dagger F$. 
	
	$(2)\Leftrightarrow (1)$. The argument given above proves that the relative left adjoint $L'\colon\K\to\P^\dagger\A$, when it exists, is given by $L'X=\K(X,F-)$. Thus the equivalence between $(1)$ and $(2)$ follows at once by definitions of $\P^\dagger\A$ and virtual left adjoints.
\end{proof}

Next we consider the notion of virtual colimit in a $\V$-category. Recall that, given a weight $M\colon\C^{op}\to\V$ and a $\V$-functor $H\colon\C\to\A$, we say that the colimit of $H$ weighted by $M$ exists in $\A$ if the $\V$-functor $[\C^{op},\V](M,\A(H,-))\colon\A\to\V$ is representable. Similarly: 

\begin{Def}
	Given a $\V$-category $\A$, a weight $M\colon\C^{op}\to\V$ with small domain, and $H\colon\C\to\A$, we say that the {\em virtual colimit} of $H$ weighted by $M$ exists in $\A$ if $[\C^{op},\V](M,\A(H,-))$ is a small $\V$-functor. We say that $\A$ is {\em virtually cocomplete} if it has all virtual colimits.
\end{Def}

\begin{prop}\label{virtualcol}
	Given $H$ and $M$ as above, the virtual colimit of $H$ weighted by $M$ exists in $\A$ if and only if the colimit $M*ZH$ exists in $\P^\dagger\A$. In this case $$[\C^{op},\V](M,\A(H,-))\cong M*ZH.$$
\end{prop}
\begin{proof}
	Consider the $\V$-functor $X:=[\C^{op},\V](M,\A(H,-))\colon\A\to\V$, if the virtual colimit of $H$ weighted by $M$ exists in $\A$ then $X$ is small  (by definition) and
	\begin{equation*}
		\begin{split}
			\P^\dagger\A(X,ZA)&\cong XA\\
			&\cong [\C^{op},\V](M,\A(H,A))\\
			&\cong [\C^{op},\V](M,\P^\dagger\A(ZH,ZA))\\
		\end{split}
	\end{equation*}
	for any $A\in\A$. Since the representables are codense in $\P^\dagger\A$ it follows that $X\cong M*ZH$ exists in $\P^\dagger\A$. Conversely, if $M*ZH$ exists in $\P^\dagger\A$ then the same chain of isomorphism above shows that it is isomorphic to $[\C^{op},\V](M,\A(H,-))$, which is then small. Therefore the virtual colimit of $H$ weighted by $M$ exists in $\A$.
\end{proof}

It follows that $\A$ is virtually cocomplete if and only if $\P^\dagger\A$ has all colimits of representables; this is equivalent to $\P^\dagger\A$ actually being cocomplete by \cite[Theorem~3.8]{DL07}. 

\begin{prop}
	Let $\A$ be a $\V$-category and $Y\colon\A\to\P\A$ be the inclusion. Then $\A$ is virtually cocomplete if and only if for any small $F\colon\A^{op}\to\V$ --- that is, for any $F\in\P\A$ --- the $\V$-functor $\P\A(F,Y-)\colon\A\to\V$ is also small.
\end{prop}
\begin{proof}
	Given a small $\V$-category $\C$ and $\V$-functors $M\colon\C\to\V$ and $H\colon\C\to\A$, then
		$$ [\C^{op},\V](M,\A(H,-))\cong \P\A(F,Y-) $$
	where $F:=\tx{Lan}_{H^{op}}M$. To conclude it's then enough to recall that the virtual colimit of $H$ weighted by $M$ exists in $\A$ if and only if $ [\C^{op},\V](M,\A(H,-))$ is small, and then note that a $\V$-functor $F\colon\A^{op}\to\V$ is small if and only if it is the left Kan extension of its restriction to some small full subcategory of $\A^{op}$.
\end{proof}

Next we can prove the following (see also \cite[Remark~3.5]{DL07}):

\begin{prop}\label{accvirtcoco}
	A $\V$-category which is accessible or conically accessible is virtually cocomplete.
\end{prop}
\begin{proof}
	Since every accessible $\V$-category is conically accessible, it's enough to prove the proposition for a conically accessible $\V$-category. If $\A$ is conically accessible then $\P^\dagger\A$, seen as a full subcategory of $[\A,\V]^{op}$, consists of the conically accessible presheaves (Proposition~\ref{acc=small}). Since these are closed under small colimits in $[\A,\V]^{op}$, it follows that $\P^\dagger\A$ is cocomplete.
\end{proof}

Virtually cocomplete $\V$-categories behave well with respect to virtual left adjoints:

\begin{prop}\label{subvref}
	Let $\A$ be virtually cocomplete and $H\colon \G\hookrightarrow\A$ be a full subcategory; then the induced $\V$-functor $\A(H,1)\colon \A\to[\G^{op},\V]$ has a virtual left adjoint.
\end{prop}
\begin{proof}
	Let $\K=[\G^{op},\V]$ and $F=\A(H,1)$; we need to prove that for each $X\in\K$ the functor $\K(X,F-)$ is small. If $X=\G(-,G)$ for some $G\in\G$ then 
	\begin{equation*}
		\begin{split}
			\K(X,F-)&\cong [\G^{op},\V](\G(\square,G),\A(H\square,-))\\
			&\cong \A(G,-);\\
		\end{split}
	\end{equation*}
	hence $\K(X,F-)\cong\A(G,-)$ is small. If $X$ is any presheaf on $\G$ then we can write it as a weighted colimit of representables; thus $\K(X,F-)$ will be a limit of small functors by the argument above. In other words $\K(X,F-)$ is a colimit of elements in $\P^\dagger\A$ and hence, since $\P^\dagger\A$ is cocomplete by hypothesis, is small. It follows that $F$ has a virtual left adjoint.
\end{proof}

The following characterizes accessible $\V$-functors between accessible $\V$-categories. 

\begin{prop}\label{virtual=accessible}
	Let $F\colon \A\to\K$ be a $\V$-functor between accessible $\V$-categories; the following are equivalent:
	\begin{enumerate}\setlength\itemsep{0.25em}
		\item $F$ has a virtual left adjoint;
		\item $F$ is accessible.
	\end{enumerate}
	Moreover if $\A,\K$, and $F$ are $\alpha$-accessible, the virtual left adjoint restricts to the $\alpha$-presentable objects: if $L\dashv \P^\dagger F$ then $L$ restricts to $L_\alpha\colon  \P^\dagger (\K_{\alpha})\to \P^\dagger (\A_{\alpha})$. 
	
	A corresponding statement holds in the case of conically accessible categories: $\A,\K$, and $F$ are assumed to be conically ($\alpha$-)accessible instead of ($\alpha$-)accessible and the full subcategories $\A_{\alpha}$ and $\K_{\alpha}$ are replaced by $\A^c_{\alpha}$ and $\K^c_{\alpha}$.
\end{prop}
\begin{proof}
	A virtual left adjoint to $F$ exists if and only if $\K(X,F-)$ is small for each $X\in\K$, and this, by Proposition~\ref{acc=small}, is the same as saying that $\K(X,F-)$ is an accessible $\V$-functor for each $X$ in $\K$. 
	
	Suppose that $F$ is accessible, take then $X\in\K$, and consider $\alpha$ such that $X$ is $\alpha$-presentable and $F$ preserves $\alpha$-flat colimits. Then $\K(X,F-)$ preserves $\alpha$-flat colimits as well and hence is small. Vice versa, assume that each $\K(X,F-)$ is accessible. Let $\alpha$ be such that $\K$ is $\alpha$-accessible and consider $\beta\geq\alpha$ such that for each $X\in\K_{\alpha}$ the $\V$-functor $\K(X,F-)$ preserves all $\beta$-flat colimits. Since $\K_{\alpha}$ is a strong generator made of ($\alpha$- and hence) $\beta$-presentable objects, it follows that $F$ preserves $\beta$-flat colimits as well.
	
	Regarding the assertion that if $\A,\K$, and $F$ are $\alpha$-accessible the virtual left adjoint restricts to the $\alpha$-presentable, it's enough to note that for any $X\in\K_{\alpha}$
	$$\K(X,J-)\cong \tx{Lan}_H\K(X,JH-),$$ 
	since $\K(X,J-)$ preserves $\alpha$-flat colimits, where $H\colon \A_{\alpha}\hookrightarrow\A$ is the inclusion. This means exactly that the left adjoint $L\colon \P^\dagger (\K)\to\P^\dagger (\A)$ restricts to $ L_{\alpha}\colon\P^\dagger (\K_{\alpha})\to \P^\dagger (\A_{\alpha})$ as desired.
	
	The same proof applies in the conically accessible case. 
\end{proof}

\begin{obs}
	Note that in the first part of the previous proposition it is enough to ask that each object of $\K$ is presentable, instead of $\K$ being accessible.
\end{obs}

An immediate consequence is:

\begin{cor}\label{accvirtref}
	Let $\A$ be an accessible and accessibly embedded subcategory of an accessible $\V$-category $\K$; then $\A$ is virtually reflective in $\K$. In fact it suffices that $\A$ be conically accessible and conically accessibly embedded.
\end{cor}
\begin{proof}
	Follows directly from Proposition~\ref{virtual=accessible} applied to the inclusion of $\A$ in $\K$.
\end{proof}

The next step is to prove the opposite direction:

\begin{prop}\label{vr-acc}
	Let $\K$ be a locally presentable $\V$-category and $J\colon \A\hookrightarrow \K$ be a virtually reflective and accessibly embedded subcategory; then $\A$ is accessible.
\end{prop}
\begin{proof}
	\textsc{Strategy.} We shall choose small full subcategories $H\colon\C\hookrightarrow\K$  and $H'\colon\D\to\A$ such that:\begin{enumerate}
		\item[(i)] $H$ is dense, so that $\K(H,1)\colon\K\hookrightarrow[\C^{op},\V]$ is fully faithful;
		\item[(ii)] $\D\subseteq \C$, with inclusion $J'\colon\D\hookrightarrow\C$, so that we have a fully faithful $\V$-functor $\tx{Lan}_{J'^{op}}\colon [\D^{op},\V]\hookrightarrow [\C^{op},\V]$;
		\item[(iii)] the intersection of these full subcategories is $\A$, as in the diagram below.
		\begin{center}
			\begin{tikzpicture}[baseline=(current  bounding  box.south), scale=2]
				
				\node (a0) at (0,0.8) {$\A$};
				\node (b0) at (1.2,0.8) {$[\D^{op},\V]$};
				\node (c0) at (0,0) {$\K$};
				\node (d0) at (1.2,0) {$[\C^{op},\V]$};
				
				\path[font=\scriptsize]
				
				(a0) edge [right hook->] node [above] {$\A(H',1)$} (b0)
				(a0) edge [right hook->] node [left] {$J$} (c0)
				(b0) edge [right hook->] node [right] {$\tx{Lan}_{J'^{op}}$} (d0)
				(c0) edge [right hook->] node [below] {$\K(H,1)$} (d0);
				
			\end{tikzpicture}	
		\end{center}
	\end{enumerate}   
	Now $[\D^{op},\V], [\C^{op},\V]$, and $\K$  are all accessible $\V$-categories (in fact locally presentable) while $\tx{Lan}_{J'^{op}}$ and $\K(H',1)$ are accessible embeddings; thus $\A$ is accessible by  Lemma~\ref{intersection-acc}. 
	
	\textsc{Main Step.} Let $\alpha$ be some regular cardinal such that $\K$ is an $\alpha$-accessible $\V$-category and the inclusion $J$ is an $\alpha$-accessible $\V$-functor. Let $\C$ be any full subcategory of $\K$ containing $\K_\alpha$ and closed under $\alpha$-small colimits, denote by $H\colon\C\hookrightarrow\K$ the inclusion. This satisfies (i). Commutativity of the square in (iii) says that the canonical maps
	$$ \K(HC,JH'-)*\A(H'-,A) \longrightarrow \K(HC,JA)$$ 
	are invertible for all $C\in\C$, $A\in\A$: the colimit on the left gives the left Kan extension $\tx{Lan}_{J'^{op}}\A(H'-,A) $ evaluated at $C$. Now $\K(HC,J-)$ is the value at $HC$ of the virtual reflection $L\colon \K\to\P^\dagger \A$, and invertibility of the above maps says that $\K(HC,J-)$ is the left Kan extension of its restriction $\K(HC,JH'-)$ to $\D$, or in other words that \begin{enumerate}
		\item[(iii')] $L$ maps $\C\subseteq\K$ into $\P^\dagger\D$. 
	\end{enumerate}
	In fact we'll see that this implies the intersection property of the square in (iii). Suppose then that $K\in\K$, and $\K(H-,K)$ is in the image of $\tx{Lan}_{J'^{op}} $, so that in fact 
	$$ \K(H-,K)\cong \tx{Lan}_{J'^{op}}\K(HJ'-,K). $$
	Since $H\colon\C\hookrightarrow\K$ is $\alpha$-cocontinuous, $\K(H-,K)$ is $\alpha$-continuous and hence $\alpha$-flat, and now by Lemma~\ref{flat-restriction} also $\K(HJ'-,K)$ is $\alpha$-flat. Thus we can form the colimit $\K(HJ'-,K)*H'\in\A$, and it will be preserved by $J$, giving
	\begin{equation*}
		\begin{split}
			J(\K(HJ'-,K)*H'))&\cong\K(HJ'-,K)*JH'\\
			&\cong \K(HJ'-,K)*HJ'\\
			&\cong \tx{Lan}_{J'^{op}}\K(HJ'-,K)*H\\
			&\cong \K(H-,K)*H\\
			&\cong K,\\
		\end{split}
	\end{equation*}
	so $K\in\A$ as required.
	
	\textsc{Choosing $\C$ and $\D$.} It remains to show that we can choose a small $\K_\alpha\subseteq\C\subseteq\K$ closed under $\alpha$-small colimits and $\D\subseteq\A \cap\C$ such that $L$ maps $\C$ into $\P^\dagger\D$. We do this by recursion on $0<i<\alpha$. Define $\C_1:=\K_{\alpha}$, and $\D_1\subseteq\A$ to be small and such that $L(\C_1)\subseteq\P^\dagger (\D_1)$ (this exists since $\C_1$ is small). Now, given $0<i<\alpha$, and the small $\V$-categories $\C_i$ and $\D_i$, we define $\C_{i+1}$ to be the closure of $J\D_i$ in $\K$ under $\alpha$-small colimits, and $\D_{i+1}\subseteq\A$ to be such that $\D_{i}\subseteq\D_{i+1}$ and $L(\C_{i+1})\subseteq\P^\dagger (\D_{i+1})$. Take unions at the limit steps and then define $\D:=\cup_{i<\alpha}\D_i$ and $\C:=\cup_{i<\alpha}\C_i$. Then $L$ maps $\C$ into $\P^\dagger\D$ by construction and, by regularity of $\alpha$, each $\alpha$-small diagram in $\C$ factors through some $\C_{i+1}$ which is closed in $\K$ under $\alpha$-small colimits by construction; hence $\C$ is closed under them as well. 
\end{proof}

The same holds if the ambient $\V$-category $\K$ is just accessible:

\begin{cor}\label{vr-acc2}
	Let $\K$ be an accessible $\V$-category and $J\colon \A\hookrightarrow \K$ be virtually reflective and accessibly embedded; then $\A$ is accessible.
\end{cor}
\begin{proof}
	Since $\K$ is accessible then we can find a small $\C$ and an accessible embedding $\K\hookrightarrow[\C,\V]$ which has a virtual left adjoint by Corollary~\ref{accvirtref}. Since virtual adjoints compose we can now apply Proposition~\ref{vr-acc} to conclude that $\A$ is accessible.
\end{proof}

It is not true in general that every $\alpha$-accessibly embedded and virtually reflective subcategory of a locally $\alpha$-presentable $\V$-category, is $\alpha$-accessible. However 
this holds with a further assumption:

\begin{cor}
	Let $\K$ be locally presentable and $J\colon \A\hookrightarrow \K$ be virtual reflective, $\alpha$-accessibly embedded, and such that the virtual reflection $L\colon \P^\dagger \K\to\P^\dagger \A$ restricts to $ L_{(\alpha)}\colon \P^\dagger (\K_{\alpha})\to \P^\dagger (\A_{(\alpha)})$, where $\A_{(\alpha)}:=\A\cap\K_{\alpha}$. Then $\A$ is $\alpha$-accessible.
\end{cor}
\begin{proof}
	In the proof of Proposition~\ref{vr-acc} we can choose $\C=\K_\alpha$ and $\D=\A_{(\alpha)}\subseteq\A_\alpha$. It follows by the intersection property in (iii) that $\A_{(\alpha)}$ is dense; moreover the weights $\A(H'-,A)$ are all $\alpha$-flat, since $\tx{Lan}_{J'^{op}}\A(H'-,A)\cong \K(H-,JA)$ is $\alpha$-continuous. Therefore it follows that $\A$ is $\alpha$-accessible.
\end{proof}

\subsection{Virtual orthogonality}\label{vir-orthogonal}

Next we introduce the third and last virtual concept of this paper; this is a generalization of the more common notion of orthogonality.

\begin{Def}
	Let $\K$ be a $\V$-category, $Z\colon \K\hookrightarrow\P^\dagger\K$ be the inclusion, and $f\colon ZX\to P$ a morphism in $\P^\dagger\K$ with representable domain. We say that an object $A$ of $\K$ is {\em orthogonal with respect to $f$} if
	$$ \P^\dagger\K(f,ZA)\colon \P^\dagger\K(P,ZA)\longrightarrow\P^\dagger\K(ZX,ZA) $$
	is an isomorphism in $\V$; in other words if $ZA\in\P^\dagger\K$ is orthogonal with respect to $f$.  
\end{Def}

Let us unwind this definition. Given an object $P$ in $\P^\dagger\K$, we can write it as a limit of representables $P\cong \{ M ,ZH\}$. Thus to give $f\colon ZX\to P$ is the same as giving a cylinder $\bar{f}\colon  M \to \K(X,H-)$; moreover $\P^\dagger\K(ZX,ZA)\cong \K(X,A)$ and $\P^\dagger\K(P,ZA)\cong  M *\K(H-,A)$. As a consequence, an object $A$ of $\K$ is orthogonal with respect to $f\colon ZX\to P$ if and only if the map
$$  M *\K(H-,A)\to \K(X,A)$$
induced by $\bar{f}\colon  M \to \K(X,H-)$ is an isomorphism.

When $P=ZY$ is representable we recover the usual notion of orthogonality.

\begin{Def}\label{virtualorth}
	Let $\K$ be a $\V$-category and $\M$ be a small collection of morphisms in $\P^\dagger\K$ of the form $f\colon ZX\to P$. We denote by $\M^\perp$ the full subcategory of $\K$ spanned by the objects which are orthogonal with respect to each $f\in\M$. We call {\em virtual orthogonality class} any full subcategory of $\K$ which arises in this way.
\end{Def}

\begin{obs}
	In the ordinary context, an equivalent for of this notion was already considered by Guitart and Lair in \cite[Section~4]{GL81:articolo}. Given a cone $c\colon\Delta X\to H$ in a category $\K$, they say that an object $A\in\K$ ``satisfies'' the cone $c$ if $\K(c,A)$ induces an isomorphism
	$$ \K(X,A)\cong \colim\K(H-,A). $$
	It's easy to see that, considering $P:=\lim YH\in\P^\dagger\K$ and the induced map $\bar c\colon ZX\to P$, an object $A$ of $\K$ satisfies $H$ if and only if it is orthogonal with respect to $\bar c$. In \cite[Section~4]{GL81:articolo} they prove then that each sketchable category $\tx{Mod}(\S)$ is (in our terminology) a virtual orthogonality class in its ambient category $[\C,\bo{Set}]$, but not the vice versa.
\end{obs}

The next results can be seen as the analogue of the relation between locally presentable categories and orthogonality classes.

\begin{prop}\label{acc-virort}
	Let $\K$ be accessible and $J\colon \A\hookrightarrow\K$ be accessible and accessibly embedded; then $\A$ is a virtual orthogonality class in $\K$.
\end{prop}
\begin{proof}
	Consider a regular cardinal $\alpha$ for which $\A,\K$, and $J$ are $\alpha$-accessible and $J$ preserves the $\alpha$-presentable objects (Corollary~\ref{preserving-flat-presentables}); this implies in particular that $\A_\alpha=\A\cap\K_\alpha$. Denote the inclusions as below.
	\begin{center}
		\begin{tikzpicture}[baseline=(current  bounding  box.south), scale=2]
			
			\node (a0) at (0,1) {$\A$};
			\node (b0) at (1.1,1) {$\K$};
			\node (c0) at (0,0) {$\A_{\alpha}$};
			\node (d0) at (1.1,0) {$\K_{\alpha}$};
			
			\node (a) at (2.2,1) {$\P^\dagger (\A)$};
			\node (b) at (3.6,1) {$\P^\dagger (\K)$};
			\node (c) at (2.2,0) {$\A$};
			\node (d) at (3.6,0) {$\K$};
			\node (e) at (2.9,1.15) {$\perp$};
			
			\path[font=\scriptsize]
			
			(a0) edge [right hook->] node [below] {$J$} (b0)
			(c0) edge [right hook->] node [left] {$H$} (a0)
			(d0) edge [right hook->] node [right] {$H'$} (b0)
			(c0) edge [right hook->] node [below] {$J_{\alpha}$}(d0)
			
			(a) edge [bend left, <-] node [above] {$L$} (b)
			(a) edge [right hook->] node [below] {$\P^\dagger J$} (b)
			(c) edge [right hook->] node [left] {$W$} (a)
			(d) edge [right hook->] node [right] {$Z$} (b)
			(c) edge [right hook->] node [below] {$J$}(d);
			
		\end{tikzpicture}	
	\end{center}
	We wish to show that $\A$ can be identified with the virtual orthogonality class defined by the set
	$$ \M:=\{\eta_X\colon ZX\to (\P^\dagger J)LZX |\ X\in\K_{\alpha} \} $$
	where each $\eta_X\colon ZX\to (\P^\dagger J)LZX$ is the component at $X$ of the unit of the adjunction.\\
	On one hand, given any $A\in\A$, the object $JA$ is orthogonal with respect to each $\eta_X$ in $\M$; in fact the orthogonality condition holds with respect to $\eta_X$ for any $X\in\K$:
	\begin{equation*}
		\begin{split}
			\P^\dagger \K(ZX,ZJA)&\cong\P^\dagger \K(ZX,(\P^\dagger J)WA)\\
			&\cong \P^\dagger \A(LZX,WA)\\
			&\cong \P^\dagger \K((\P^\dagger J)LZX,(\P^\dagger J)WA)\\
			&\cong \P^\dagger \K((\P^\dagger J)LZX,ZJA).\\
		\end{split}
	\end{equation*}
	
	Conversely suppose that $Y\in\K$ is orthogonal with respect to $\eta_X$ for each $X\in\K_{\alpha}$. Note first that
	\begin{equation*}
		\begin{split}
			(\P^\dagger J)LZX&\cong\tx{Lan}_J\K(X,J-)\\
			&\cong \tx{Lan}_J\tx{Lan}_H\K(X,JH-)\\
			&\cong \tx{Lan}_{JH}\K(X,JH-)\\
			&\cong \{\K(X,JH-),ZJH\},\\
		\end{split}
	\end{equation*}
	where the second isomorphism follows from the fact that $\K(X,J-)$ preserves $\alpha$-flat colimits and $\A$ is $\alpha$-accessible. As a consequence we obtain that for each $X\in\K_{\alpha}$:
	\begin{align}
			\K(X,Y)&\cong\P^\dagger \K((\P^\dagger J)LZX,ZY)\\
			&\cong \P^\dagger \K(\{\K(X,JH-),ZJH\},ZY)\nonumber\\
			&\cong \K(X,JH-)*\K(JH-,Y)\nonumber\\
			&\cong \K(X,H'J_{\alpha}-)*\K(H'J_{\alpha}-,Y)\nonumber\\
			&\cong \K_{\alpha}^{op}(J_{\alpha}-,X)*\K(H'J_{\alpha}-,Y)\nonumber\\
			&\cong \tx{Lan}_{J_{\alpha}^{op}}\K(H'J_{\alpha}-,Y)(X)\nonumber
	\end{align}
	where $(2)$ holds because $Y$ is orthogonal with respect to $\eta_X$. It follows that $$\K(H'-,Y)\cong\tx{Lan}_{J_{\alpha}^{op}}\K(H'J_{\alpha}-,Y),$$ but the weight $\K(H'-,Y)$ is $\alpha$-flat; therefore $\K(JH-,Y)\cong\K(H'J_{\alpha}-,Y)$ is $\alpha$-flat as well by Lemma~\ref{flat-restriction}. As a consequence, for each $X\in\K_{\alpha}$
	\begin{equation*}
		\begin{split}
			\K(X,Y)&\cong\K(JH-,Y)*\K(X,JH-)\\
			&\cong \K(X,\K(JH-,Y)*JH)\\
			&\cong \K(X,J(\K(JH-,Y)*H));\\
		\end{split}
	\end{equation*}
	thus $Y\cong J(\K(JH-,Y)*H)$ lies in $\A$. 
\end{proof}

Conversely, the accessibility of virtual orthogonality classes can be obtained as a consequence of the following. 

\begin{prop}\label{virorth-sketch}
	Each virtual orthogonality class $\A$ of a presheaf $\V$-category $\K=[\C,\V]$ is equivalent to the category of models of a sketch. More precisely: there exists a fully faithful $J\colon \C\hookrightarrow\B$ and a sketch $\S=(\B,\mathbb{L},\mathbb{C})$ on $\B$ such that $\tx{Ran}_J$ induces an equivalence $\A\simeq\tx{Mod}(\S)$.
\end{prop}
\begin{proof}
	Let $J\colon \A\hookrightarrow\K$ be a virtual orthogonality class in $\K$ defined by a set of morphisms $\M$. Without loss of generality we can assume that $\M$ consists of a single arrow $f\colon ZX\to P=\{ M ,ZH\}$ with $X\in\K$, $Z\colon \K\hookrightarrow\P^\dagger \K$ being the inclusion, $ M \colon \D\to\V$ a weight, and $H\colon \D\to\K$ a diagram in $\K$.
	Consider now the closure $\B^{op}$ of $\C^{op}\hookrightarrow\K=[\C,\V]$ under $\alpha$-small colimits, where $\alpha$ is such that $\B$ contains $X$ and the image of $H$; let $H'\colon \D^{op}\to\B$ be the induced map. In particular $\B$ is the free completion of $\C$ under $\alpha$-small limits; so that right Kan extending along the inclusion induces an equivalence
	$$W\colon [\C,\V]\longrightarrow\alpha\tx{-Cont}[\B,\V].$$ 
	Note moreover that $f$ corresponds to a cylinder $\bar{f}\colon  M \to\K(X,H-)\cong\B(H'-,X)$.
	
	Now, a $\V$-functor $A\in\K$ is orthogonal with respect to $f$ if and only if
	$$ \P^\dagger \K(f,A)\colon \P^\dagger \K(P,ZA)\to\P^\dagger \K(ZX,ZA) $$
	is an isomorphism; but on one hand we have
	\begin{equation*}
		\begin{split}
			\P^\dagger \K(ZX,ZA)&\cong\K(X,A)\\
			&\cong [\B,\V](WX,WA)\\
			&\cong [\B,\V](\B(X,-),WA) \\
			&\cong (WA)(X)\\
		\end{split}
	\end{equation*}
	on the other
	\begin{equation*}
		\begin{split}
			\P^\dagger \K(P,ZA)&\cong M *\K(H-,A)\\
			&\cong  M *[\B,\V](WH-,WA)\\
			&\cong  M *[\B,\V](\B(H',-),WA) \\
			&\cong  M *(WA)H'\\
		\end{split}
	\end{equation*}
	It follows then that $A$ is orthogonal with respect to $f$ if and only if $WA$ sends the cylinder $\bar{f}$ to a colimiting cylinder. In conclusion $\A$ is equivalent to the full subcategory of $[\B,\V]$ given by the $\alpha$-continuous functors which send $\bar{f}$ to a colimiting cylinder, and this is the $\V$-category of models of a sketch on $\B$.
\end{proof}

One can also show the converse: every $\V$-category of models of a sketch is a virtual orthogonality class in its ambient category. This can of course be seen as a consequence of the fact that sketchable implies accessible (Theorem~\ref{sketches-accessible}) which in turn implies virtual orthogonality class (Proposition~\ref{acc-virort}), but we can also provide a direct proof:

\begin{prop}
	Let $\S=(\B,\mathbb{L},\mathbb{C})$ be a sketch; then $\tx{Mod}(\S)$ is a virtual orthogonality class in $[\B,\V]$.
\end{prop}
\begin{proof}
	Let $\K:=[\B,\V]$, $Y\colon\B^{op}\to\K$ be the Yoneda embedding, and $Z\colon\K\hookrightarrow\P^\dagger\K$ be the inclusion into the free completion. It's enough to show that each cylinder $c\in\mathbb{L}$ there exists a morphism $f_c\colon ZX\to P$ in $\P^\dagger \K$ for which a functor $F\colon \B\to\V$ sends $c$ to a limiting cylinder if and only if $F$ is orthogonal with respect to $f_c$; plus the colimit version of this for any $d\in\mathbb{C}$.
	
	In the limit case, the virtual orthogonality notion coincides with standard orthogonality. Let $c\colon N\to\B(B,H-)$ be a cylinder in $\mathbb{L}$; then we can consider $X:=N*YH^{op}\in\K$ and the map $f_c\colon X\to YB$ induced by $c$. Now, for each $F\colon\B\to\V$, since by construction $\K(X,F)\cong \{N,FH-\}$, it follows that $F$ sends $c$ to a limiting cylinder if and only if it is orthogonal with respect to $f_c$, and this is the same as virtual orthogonality with respect to $Zf_c$.
	
	For the colimit case, consider $d\colon M\to \B(K-,C)$ in $\mathbb{C}$ and define $P:=\{M,ZYK^{op}\}$ in $\P^\dagger\K$; then $d$ induces a map $f_d\colon ZYC\to P$. Now note that, for each $F\colon\B\to\V$,
	$$ \P^\dagger\K(ZYC,ZF)\cong [\B,\V](YC,F)\cong FC $$
	and on the other hand
	$$ \P^\dagger\K(P,ZF)\cong M*\P^\dagger\K(ZYK^{op}-,ZF)\cong M*[\B,\V](YK^{op}-,F)\cong M*FK. $$
	Thus it follows again that $F$ sends $d$ to a colimiting cylinder if and only if it is orthogonal with respect to $f_d$.
\end{proof}

\subsection{The characterization theorems}\label{virtual-main}

We can now sum up all the results above in the characterization Theorem below.

\begin{teo}\label{accessible-embedded}
	For an accessible $\V$-category $\K$ and a fully faithful inclusion $\A\hookrightarrow\K$, the following are equivalent:\begin{enumerate}\setlength\itemsep{0.25em}
		\item $\A$ is accessible and accessibly embedded;
		\item $\A$ is accessibly embedded and virtually reflective;
		\item $\A$ is a virtual orthogonality class.
	\end{enumerate}
\end{teo}
\begin{proof}
	$(1)\Leftrightarrow(2)$ are given by the Corollaries~\ref{accvirtref} and~\ref{vr-acc2}. $(1)\Rightarrow(3)$ is a consequence of Proposition~\ref{acc-virort}. For the implication $(3)\Rightarrow(1)$, it follows from Propositions~\ref{virorth-sketch} and~\ref{sketches-accessible} that $\A$ is accessible; moreover it is easily seen to be closed under $\alpha$-flat colimits in $\K$, where $\alpha$ is such that all the morphisms in the family defining the virtual orthogonality class lie in $\P^\dagger(\K_\alpha)$.
\end{proof}

In general we obtain:

\begin{teo}\label{accessible-char}
	The following are equivalent for a $\V$-category $\A$: \begin{enumerate}\setlength\itemsep{0.25em}
		\item $\A$ is accessible;
		\item $\A\simeq \alpha\tx{-Flat}(\C,\V)$ for some $\alpha$ and some small $\C$;
		\item $\A$ is accessibly embedded and virtually reflective in $[\C,\V]$ for some small $\C$;
		\item $\A$ is a virtual orthogonality class in $[\C,\V]$ for some small $\C$;
		\item $\A$ is equivalent to the $\V$-category of models of a sketch.
	\end{enumerate}
\end{teo}
\begin{proof}
	The equivalences $(1)\Leftrightarrow(3)\Leftrightarrow(4)$ are a direct consequence of Theorem~\ref{accessible-embedded}.  The implication $(1)\Leftrightarrow (2)$ is given by Proposition~\ref{acc=flat}, while $(1)\Leftrightarrow (5)$ is given by Theorem~\ref{sketches-accessible}.
\end{proof}

A few consequences of these characterization theorems are:

\begin{cor}
	If $\A$ is a $\V$-category with $\alpha$-flat colimits, for some $\alpha$, then $\A$ is accessible if and only if it is virtually cocomplete and has a dense generator made of presentable objects.
\end{cor}
\begin{proof}
	By the previous theorem it's enough to prove that $\A$ is accessibly embedded and virtually reflective in some category of presheaves. Let $\C\subseteq\A$ be a dense presentable generator; then the inclusion $\A\hookrightarrow[\C^{op},\V]$ is accessible (since every object of $\C$ is presentable) and virtually reflective by Proposition~\ref{subvref}.
\end{proof}

\begin{cor}
	If $\A$ is a $\V$-category with $\alpha$-flat colimits, for some $\alpha$, then $\A$ is accessible if and only if it has a dense generator and $\P^\dagger \A$ consists exactly of the accessible presheaves out of $\A$.
\end{cor}
\begin{proof}
	Note that each object of $\A$ is presentable, since representable functors are accessible by hypothesis; moreover $\P^\dagger \A$ is cocomplete, since accessible functors are limit closed in $[\A,\V]$. Thus the result follows from the previous corollary.
\end{proof}

So far we have only given a characterization of the accessible $\V$-categories, but not of the conically accessible ones; this is what we can say in that context:

\begin{teo}\label{con-accessible-embedded}
	For a conically accessible $\V$-category $\K$ and a fully faithful inclusion $\A\hookrightarrow\K$, the following are equivalent:\begin{enumerate}\setlength\itemsep{0.25em}
		\item $\A$ is conically accessible and conically accessibly embedded;
		\item $\A_0$ is accessible and accessibly embedded in $\K_0$;
		\item $\A_0$ is accessibly embedded and virtually reflective in $\K_0$;
		\item $\A_0$ is a virtual orthogonality class in $\K_0$.
	\end{enumerate}
\end{teo}
\begin{proof}
	Follows from Theorem~\ref{accessible-embedded} for $\V=\bo{Set}$ and Corollary~\ref{conically-acc}.
\end{proof}

And as a consequence: 

\begin{teo}\label{con-accessible-char}
	The following are equivalent for a $\V$-category $\A$:\begin{enumerate}\setlength\itemsep{0.25em}
		\item $\A$ is conically accessible;
		\item $\A$ is conically accessibly embedded in $[\C,\V]$, for some small category $\C$, and $\A_0$ is virtually reflective in $[\C,\V]_0$;
		\item $\A$ is a full subcategory of $[\C,\V]$, for some small category $\C$, and $\A_0$ accessible and accessibly embedded in $[\C,\V]_0$;
		\item $\A$ is a full subcategory of $[\C,\V]$, for some small category $\C$, and $\A_0$ is a virtual orthogonality class in $[\C,\V]_0$.
	\end{enumerate}
\end{teo}

We don't yet know whether a $\V$-category which is conically accessibly embedded  and virtually reflective in some $[\C,\V]$, is also conically accessible. The issue is that the fact that a $\V$-category $\A$ is virtually reflective in $\K$ is not known to imply that $\A_0$ is virtually reflective in $\K_0$.

\vspace{6pt}
\subsection{Cone-reflectivity and cone-injectivity}\label{cre+p}

In this section we go back to the ordinary setting ($\V=\bo{Set}$) and compare the virtual concepts introduced in Section~\ref{virt-reflective} and \ref{vir-orthogonal} with those of cone-reflectivity and cone-injectivity class from \cite{AR94:libro}. For that we first need to recall the notions of petty and lucid functors introduced by Freyd:

\begin{Def}[\cite{Fre69}]
	Let $\A$ be a category; a functor $P\colon \A\to\bo{Set}$ is called {\em petty} if there exists a family $(A_i)_{i\in I}$ in $\A$ and an epimorphism
	$$ \sum_{i\in I}\A(A_i,-)\twoheadrightarrow P. $$
	Denote by $\tx{Pt}(\A)$ the full subcategory of $[\A^{op},\bo{Set}]$ spanned by the petty functors.
\end{Def}

	Clearly every small functor is petty since every small colimit of representables is in particular a coequalizer of coproducts of them. Thus we have a fully faithful inclusion $\P\A\hookrightarrow\tx{Pt}(\A)$ as full subcategories of $[\A^{op},\bo{Set}]$; moreover the category $\tx{Pt}(\A)$ is locally small and, if we allow some colimits to be large, it can be seen as some kind of free cocompletion of $\A$:

\begin{obs}
	Let us say that a category $\L$ is {\em well cocomplete} if it is cocomplete and has all (possibly large) cointersections of regular epimorphisms. A functor $F\colon \L\to\K$ is well cocontinuous if it is cocontinuous and preserves all the cointersections of regular epimorphisms. Then, for any category $\A$, it's easy to see that $\tx{Pt}(\A)$ is well cocomplete and also the free {\em well cocompletion} of $\A$: for any well cocomplete category $\B$, precomposition with the inclusion $V\colon \A\hookrightarrow\tx{Pt}(\A)$ induces an equivalence between $[\A,\B]$ and the category of well cocontinuous functors $\tx{Pt}(\A)\to \B$.
\end{obs}

\begin{Def}[\cite{Fre69}]
	We say that a functor $L\colon\A\to\bo{Set}$ is {\em lucid} if it is petty and for any other petty $P$ and $f,g\colon P\to L$, the equalizer of $(f,g)$ is still petty. Denote by $\tx{Lcd}(\A)$ the full subcategory of $\tx{Pt}(\A)$ given by the lucid functors.
\end{Def}

Note that, thanks to \cite[Proposition~1.1]{Fre69}, in the definition above we can assume $P$ to be representable. As a consequence $L$ is lucid if and only if it is petty and for any representable $A\in\A$ and $f,g\colon \A(A,-)\to L$, the equalizer of $(f,g)$ is still petty.

\begin{obs}
	To define enriched notions of petty and lucid functors one would need to choose a suitable class of maps in $\V$ to play the role of epimorphisms in $\bo{Set}$.
\end{obs}

Every lucid functor is petty (by definition), however in general lucid and small functors are not comparable. This changes if the following conditions are satisfied: 

\begin{prop}\label{petty-lucis-small}
	The following are equivalent for a category $\A$: \begin{enumerate}\setlength\itemsep{0.25em}
		\item $\tx{Pt}(\A)$ has limits of representables (i.e. $\A$ is pre-complete);
		\item $\tx{Lcd}(\A)$ is complete and contains the representables;
		\item $\P\A$ has limits of representables;
		\item $\P\A$ is complete;
	\end{enumerate}
	and in that case $\P\A=\tx{Lcd}(\A)$.
\end{prop}
\begin{proof}
	$(1)\Leftrightarrow (2)$ is \cite[Theorem~1.(12)]{Fre69}. The fact that this implies $\P\A=\tx{Lcd}(\A)$ is \cite[Lemma~1]{Ros99} (see note below their proof), and $(2)\Rightarrow (3)$ can be seen as a consequence. $(3)\Leftrightarrow (4)$ is \cite[Theorem~3.8]{DL07} and finally $(3)\Rightarrow (1)$ is trivial since $\P\A\subseteq \tx{Pt}(\A)$.
\end{proof}

Taking the dual notions in the statement above, the proposition says in particular that a category is pre-cocomplete $(1)$ if and only if it is virtually cocomplete $(3)$.

In the next part of the section we keep working with the dual notions, $\tx{Pt}^\dagger(\A)$ and $\P^\dagger\A$. As we considered virtual left adjoints (relatively to small functors), one could introduce an adjointness condition with respect to petty functors by imposing that, for $F\colon \A\to\K$, the functors $\K(X,F-)\colon \A\to\bo{Set}$ are petty for any $X\in\K$. In other words, this is saying that $F\colon\A\to\K$ has a relative left adjoint with respect to the inclusion $\A\hookrightarrow\tx{Pt}^\dagger \A$.

The condition above turns out to be the same as $F$ satisfying the {\em solution-set condition}: indeed $\K(X,F-)$ is petty if and only if there exists an epimorphism $ \textstyle\sum_{i\in I}\A(A_i,-)\twoheadrightarrow \K(X,F-)$, for some $A_i\in\A$, if and only if there exists a cone $(h_i\colon X\to FA_i)_{i\in I}$ such that any map $h\colon X\to FA$ factors as $h=F(f)\circ h_i$ for some $i\in I$ and $f$ in $\A$; this is exactly the solution-set condition for $F$.

\begin{Def}[\cite{AR94:libro}]
	We say that  a fully faithful inclusion $J\colon\A\hookrightarrow\K$ is {\em cone-reflective} if $J$ satisfies the solution-set condition.
\end{Def}

For fully faithful functors, having a virtual left adjoint or satisfying the solution-set condition is almost the same, at least in the virtually cocomplete context:

\begin{prop}\label{virt=cone}
	Given a virtually cocomplete category $\K$ and a fully faithful functor $J\colon \A\hookrightarrow\K$, the following are equivalent: \begin{enumerate}\setlength\itemsep{0.25em}
		\item $\A$ is cone-reflective in $\K$;
		\item $\A$ is virtually reflective in $\K$.
	\end{enumerate}
\end{prop}
\begin{proof}
	$(2)\Rightarrow(1)$ is trivial since every small functor is petty.
	
	$(1)\Rightarrow(2)$. Let $J$ have a relative left adjoint $L\colon \K\to\tx{Pt}^\dagger (\A)$, we want to prove that this actually lands in $\P^\dagger\A$. Note first that $L$ extends to a left adjoint $-\circ J\colon \tx{Pt}^\dagger (\K)\to\tx{Pt}^\dagger (\A)$ to the inclusion $\tx{Pt}^\dagger (J)\colon \tx{Pt}^\dagger (\A)\hookrightarrow\tx{Pt}^\dagger (\K)$; indeed if $P\colon \K\to\bo{Set}$ is petty then $PJ$ is covered by a family of functors of the form $\K(X,J-)$, but each of these is covered by a family of representables (from $\A$) by hypothesis, and hence $PJ$ is petty as well.\\
	Now, since $\K$ is virtually cocomplete, $\tx{Pt}^\dagger (\K)$ has colimits of representables (Proposition~\ref{petty-lucis-small}); therefore $\tx{Pt}^\dagger (\A)$ has colimits of representables as well, being reflective in $\tx{Pt}^\dagger (\K)$ with the inclusion induced by $J$. It follows again by Proposition~\ref{petty-lucis-small} that $\P^\dagger\A=\tx{Lcd}^\dagger(\A)$; therefore to prove the virtual reflectivity of $\A$ in $\K$ it's enough to show that for each $X\in\K$ the functor $LX=\K(X,J-)\colon \A\to\bo{Set}$ is lucid.\\
	Given $X\in\K$ the functor $\K(X,J-)$ is petty by cone-reflectivity of $\A$; thus we only need to show that the equalizer of any pair $f,g\colon \A(A,-)\to\K(X,J-)$ is still petty (see just below the definition of lucidity). Such a pair corresponds to maps $h,k\colon X\to JA$ in $\K$ which in turn give a pair $f',g'\colon \K(JA,-)\to\K(X,-)$ between representables in $\tx{Pt}(\K^{op})=\tx{Pt}^\dagger (\K)^{op}$. Since $\tx{Pt}^\dagger (\K)$ has colimits of representables we can consider the equalizer $P$ of $f',g'$ in $\tx{Pt}(\K^{op})$; it follows at once that $PJ$ is the equalizer of $f,g$ and this is petty because $P$ was petty and $-\circ J$ preserves petty functors.
\end{proof}

\begin{obs}
	One might think that the equivalent conditions above arise from a left adjoint to the inclusion $J\colon \P^\dagger\A\hookrightarrow\tx{Pt}^\dagger(\A)$, but that (almost) never happens. In fact $J$ has a left adjoint if and only if the two categories $\P^\dagger\A$ and $\tx{Pt}^\dagger(\A)$ coincide (since they both contain the representables, such a left adjoint is forced to be the identity), and this doesn't hold even when $\A$ is locally presentable ($\A=\bo{Ab}$ is a counterexample, see \cite{Fre69}).
\end{obs}

\begin{cor}\label{virt-sset}
	Given a fully faithful functors $J\colon\A\hookrightarrow\K$ between accessible categories, the following are equivalent:\begin{enumerate}\setlength\itemsep{0.25em}
		\item $\A$ is accessibly embedded in $\K$;
		\item $\A$ is cone-reflective in $\K$;
		\item $\A$ is virtually reflective in $\K$.
	\end{enumerate}
\end{cor}
\begin{proof}
	Put together Propositions~\ref{virt=cone} and~\ref{virtual=accessible}.
\end{proof}

The equivalence between $(1)$ and $(2)$ was first proved in \cite[Theorem~3.10]{RoTho95}. Note however that the same equivalence can't be proved in standard set theory when $J$ is replaced by any (non necessarily fully faithful) functor between accessible categories; see \cite[Theorem~6.30]{AR94:libro} and the Remark just below it. On the other hand, we have already shown that $(1)\Leftrightarrow (2)$ always holds (Proposition~\ref{virtual=accessible}).

As a corollary we recover the characterization of accessibility given in \cite{AR94:libro}:

\begin{cor}{\cite[Theorem 2.53]{AR94:libro}}
	Let $\K$ be an accessible category and $\A$ be an accessibly embedded full subcategory of $\K$. Then $\A$ is accessible if and only if it is cone-reflective and accessibly embedded in $\K$.
\end{cor}
\begin{proof}
	This is now a direct consequence of Theorem~\ref{accessible-embedded} and Proposition~\ref{virt=cone} above.
\end{proof}

Another way of characterizing accessible categories is via cone injectivity classes. We can recover this characterization using virtual orthogonality classes. 

\begin{Def}[\cite{AR94:libro}]
	Let $\K$ be a category and let $(f_i\colon X\to X_i)_{i\in I}$ be a cone in $\K$. We say that $A\in\K$ is {\em injective} with respect to the cone $(f_i)_{i\in I}$ if for any $h\colon X\to A$ there exists $i\in I$ for which $h$ factorizes through $f_i$.
\end{Def}

Equivalently, $A\in\K$ is injective with respect to the cone $(f_i)_{i\in I}$ if and only if the map
$$ \coprod_{i\in I}\K(X_i,A)\longrightarrow\K(X,A) $$
induced by the cone $(f_i)_{i\in I}$, is a surjection of sets.

\begin{Def}[\cite{AR94:libro}]
	Let $\K$ be a category and $\M$ be a small collection of cones in $\K$. We denote by $\M\tx{-inj}$ the full subcategory of $\K$ spanned by the objects which are injective with respect to each cone in $\M$. We call {\em cone-injectivity class} any full subcategory of $\K$ which arises in this way.
\end{Def}

Then cone-injectivity classes and virtual orthogonality classes turn out to be strictly related by the following:

\begin{prop}\label{cinj-vinj}
	The following hold for a given category $\K$: \begin{enumerate}\setlength\itemsep{0.25em}
		\item every cone-injectivity class in $\K$ is a virtual orthogonality class;
		\item if $\K$ has pushouts, every virtual orthogonality class in $\K$ is a cone-injectivity class.
	\end{enumerate}
\end{prop}
\begin{proof}
	$(1)$. It's enough to prove that injectivity with respect to a cone can be seen as orthogonality with respect to a suitable morphism in $\P^\dagger\K$; we are going to use the fact that a map is an epimorphism if and only if the co-diagonal out of its cokernel pair is an isomorphism. Let $(f_i\colon X\to X_i)_i$ be a cone in $\K$; consider the corresponding $f\colon ZX\to P:=\textstyle\prod_i ZX_i$ in $\P^\dagger \K$ and take the kernel pair of $f$ with the corresponding diagonal map $\delta$:
	\begin{center}
		
		\begin{tikzpicture}[baseline=(current  bounding  box.south), scale=2]
			
			\node (a) at (-0.8,0) {$ZX$};
			\node (b) at (0,0) {$P'$};
			\node (c) at (0.8,0) {$ZX$};
			\node (d) at (1.6,0) {$P$};
			
			\path[font=\scriptsize]
			
			(a) edge [->] node [above] {$\delta$} (b)
			([yshift=2pt]b.east) edge [->] node [above] {} ([yshift=2pt]c.west)
			([yshift=-1.5pt]b.east) edge [->] node [below] {} ([yshift=-1.5pt]c.west)
			(c) edge [->] node [above] {$f$} (d);
		\end{tikzpicture}
	\end{center}
	This is sent to a cokernel pair through $\P^\dagger \K(-,ZA)$ for each $A\in\K$ (since $\P^\dagger \K(-,ZA)\cong\tx{ev}_A$ is cocontinuous), with co-diagonal $\P^\dagger \K(\delta,ZA)$. As a consequence the map
	\begin{center}
		
		\begin{tikzpicture}[baseline=(current  bounding  box.south), scale=2]
			
			\node (a) at (0,0) {$\textstyle\coprod_i\K(X_i,A)\cong \P^\dagger \K(P,ZA)$};
			\node (b) at (3.2,0) {$\P^\dagger \K(ZX,ZA)\cong \K(X,A)$};

			\path[font=\scriptsize]
			
			(a) edge [->] node [above] {$\P^\dagger \K(f,ZA)$} (b);
		\end{tikzpicture}
	\end{center} 
	is an epimorphism if and only if $\P^\dagger \K(\delta,ZA)$ is an isomorphism, which in turn means that $A$ is injective with respect to $(f_i)_i$ if and only if it is orthogonal with respect to the map $\delta\colon ZX\to P'$.
	
	$(2)$. Similarly, it's enough to prove that orthogonality with respect to a map $f\colon ZX\to P$ in $\P^\dagger\K$ is the same as injectivity with respect to a set of cones. Since $P\in\P^\dagger\K$, we can find $H\colon \I\to\K$ such that $P\cong\colim ZH$. Then $f$ corresponds to a cone $(f_i\colon X\to Hi)_{i\in\I}$ over $H$ in $\K$. Now let $A\in\K$; then $A$ is orthogonal with respect to $f$ if and only if the map 
	$$ \rho\colon\colim\K(H-,A)\longrightarrow\K(X,A) $$
	induced by $(f_i)_{i\in\I}$, is bijective. Notice now that $\rho$ is surjective if and only if $A$ is injective with respect to the cone $(f_i)_{i\in\I}$; thus we only need to express the fact that $\rho$ is a monomorphism in terms of injectivity. For this observe that $\rho$ is injective if and only if, for any $h\colon X\to\A$ which factors as $h=h_i\circ f_i=h_j\circ f_j$, for some $h_i\colon Hi\to A$ and $h_j\colon Hj\to A$, there exists a zig-zag in $H/A$ connecting $h_i$ and $h_j$. Now, for any pair $i,j\in\I$ let $X_{i,j}$ be the pushout of $(f_i,f_j)$ in $\K$ and $\Xi_{ij}$ be the set of all zig-zags in $\I$ between $i$ and $j$. For any $\xi\in\Xi_{ij}$ let $X_\xi$ denote the colimit of the diagram $H(\xi)$ in $\K$ (this is obtained as a finite number of consecutive pushouts), and let $g_\xi\colon X_{i,j}\to X_\xi$ be the induced comparison; this gives a cone $(g_\xi\colon X_{i,j}\to X_\xi)_{\xi\in\Xi_{ij}}$ for any pair $i,j\in\I$. To conclude it's enough to note that to give an arrow $h\colon X_{i,j}\to A$ is the same as giving an arrow $X\to A$ which factors through $f_i$ and $f_j$; moreover $h$ factors through an arrow of the cone $(g_\xi)_{\xi\in\Xi_{ij}}$ if and only if the two factorizations of $h$ are connected by a zig-zag in $H/A$. It follows at once that $A$ is orthogonal with respect to $f$ if and only if it is injective with respect to the cones $(f_i)_{i\in\I}$ and $(g_\xi)_{\xi\in\Xi_{ij}}$ for any $i,j\in\I$.
\end{proof}

As a consequence we recover:

\begin{teo}{\cite[Theorem 4.17]{AR94:libro}}
	Let $\K$ be locally presentable and $\A$ be a full subcategory of $\K$. Then $\A$ is accessible and accessibly embedded if and only if it is a cone-injectivity class in $\K$.
\end{teo}
\begin{proof}
	Direct consequence of Theorem~\ref{accessible-embedded} and Proposition~\ref{cinj-vinj} above.
\end{proof}

\section{Limits of accessible $\V$-categories}\label{limits}

It is proved in \cite[Theorem~5.1.6]{MP89:libro} that the 2-category $\bo{Acc}$, of accessible categories, accessible functors, and natural transformations, admits all small bilimits and that the forgetful functor $U_0\colon\bo{Acc}\to\bo{CAT}$ preserves them. But more can be said: $\bo{Acc}$ has all flexible limits and $U$ preserves them (as pointed out in \cite[Remark~7.8]{BKPS89:articolo}); then one can deduce from this that $\bo{Acc}$ has all pseudolimits and bilimits (see \cite{lack20102}).

In this section we are going to show that the corresponding result holds for accessible and conically accessible $\V$-categories as well. The corresponding result for locally presentable $\V$-categories was proved in Bird's thesis \cite[Theorem~6.10]{Bir84:tesi}.

\subsection{The accessible case}

Let $\V\tx{-}\bo{Acc}$ be the 2-category of accessible $\V$-categories, accessible $\V$-functors, and $\V$-natural transformations. When $\V=\bo{Set}$ we simply write $\bo{Acc}$ instead of $\bo{Set}\tx{-}\bo{Acc}$.

\vspace{6pt}

\begin{prop}\label{acc-funtors}
	Let $\A$ be an accessible $\V$-category and $\C$ be a small $\V$-category. Then $[\C,\A]$ is accessible.
\end{prop}
\begin{proof}
	By Theorem~\ref{sketches-accessible} we can write $\A\simeq\tx{Mod}(\S)$ for a sketch $\S=(\B,\mathbb{L},\mathbb{C})$ and hence we can see $\A$ as a full subcategory of $[\B,\V]$. It follows at once that $[\C,\A]$ has a fully faithful embedding
	$$ J\colon [\C,\A]\hookrightarrow[\C,[\B,\V]]\simeq[\C\otimes\B,\V] $$
	obtained by post-composition with the inclusion of $\A$ in $[\B,\V]$. Moreover $[\C,\A]$ can be described as the full subcategory of $[\C\otimes\B,\V]$ whose objects $F$ are such that $F(c,-)\in\A$ for each $c\in\C$.
	
	Now, for each $c\in\C$ and $\lambda\colon  M \to \B(b,H-)$ in $\mathbb{L}$ consider the induced cylinder
	$$ \lambda_c\colon  M \to (\C\otimes\B)((c,b),(c,H-)) $$
	which acts constantly on $c$ in the first component; let $\mathbb{L}_\C$ be the collection of all such cylinders. Similarly, for each $c\in\C$ and $\mu\in\mathbb{C}$ define the co-cylinder $\mu_c$ accordingly, so that we can a new family $\mathbb{C}_\C$. \\
	It's easy to see that given $F\colon\C\otimes\B\to\V$, the functor $F(c,-)$ sends $\lambda$ (respectively $\mu$) to a (co)limiting cylinder if and only if $F$ sends $\lambda_c$ (respectively $\mu_c$) to a (co)limiting cylinder. As a consequence $[\C,\A]\simeq \tx{Mod}(\S')$ for the sketch
	$ \S'=(\C\otimes\B,\mathbb{L}_\C,\mathbb{C}_\C) $, and thus it is accessible by Theorem~\ref{sketches-accessible}.
\end{proof}

\begin{obs}
	In the proof above as well as in the forthcoming ones we make use of the theory of sketches; it would be interesting instead to have a proof which doesn't rely on these but we don't have one. For this reason we don't have any result on stability of $\alpha$-accessible $\V$-categories under limits in $\V\tx{-}\bo{CAT}$, for a fixed regular cardinal $\alpha$.
	Moreover we don't know whether the proposition above is still true when accessibility is replaced by conical accessibility, nonetheless we are still able to show that conically accessible $\V$-categories are stable in $\V\tx{-}\bo{CAT}$ under flexible colimits. 
\end{obs}

\begin{cor}
	Let $\A$ be an accessible $\V$-category and $\C$ be an ordinary small category. Then the power $\A^\C:=\C\pitchfork\A$ exists in $\V\tx{-}\bo{Acc}$ and is preserved by the forgetful functor $U\colon \V\tx{-}\bo{Acc}\to \V\tx{-}\bo{CAT}$.
\end{cor}
\begin{proof}
	Let $\A^\C$ be the power of $\A$ by $\C$ in $\V\tx{-}\bo{CAT}$. Since a $\V$-functor $\B\to\A^\C$ is accessible if and only if its transpose ordinary functor $\C\to\V\tx{-}\bo{CAT}(\B,\A)$ lands in $\V\tx{-}\bo{Acc}(\B,\A)$, it's enough to show that $\A^\C$ is an accessible $\V$-category. This follows at once from the previous proposition since $\A^\C\cong [\C_\V,\A]$, where $\C_\V$ is the free $\V$-category on $\C$. 
\end{proof}

The following will be needed to prove the main result of this section.

\begin{lema}\label{fullness-of-sketches}
	Let $F\colon \A_1\to\A_2$ be an accessible $\V$-functor between accessible $\V$-categories; then there exist sketches $\S_1=(\B_1,\mathbb{L}_1,\mathbb{C}_1)$ and $\S_2=(\B_2,\mathbb{L}_2,\mathbb{C}_2)$ and a $\V$-functor $K\colon \B_2\to\B_1$ for which:
	\begin{enumerate}\setlength\itemsep{0.25em}
		\item $\A_1\simeq\tx{Mod}(\S_1)$ and $\A_2\simeq\tx{Mod}(\S_2)$:
		\item The induced square
		\begin{center}
			\begin{tikzpicture}[baseline=(current  bounding  box.south), scale=2]
				
				\node (a0) at (0,0.8) {$[\B_1,\V]$};
				\node (b0) at (1.1,0.8) {$[\B_2,\V]$};
				\node (c0) at (0,0) {$\A_1$};
				\node (d0) at (1.1,0) {$\A_2$};
				
				\path[font=\scriptsize]
				
				(a0) edge [->] node [above] {$-\circ K$} (b0)
				(c0) edge [right hook->] node [left] {} (a0)
				(d0) edge [right hook->] node [right] {} (b0)
				(c0) edge [->] node [below] {$F$} (d0);
			\end{tikzpicture}	
		\end{center}
		commutes up to isomorphism.
	\end{enumerate}
	Moreover $\B_2$ can be chosen to be $(\A_2)_{\alpha}^{op}$ for arbitrarily large cardinals $\alpha$ as in~\ref{flat-to-conical}.
\end{lema}
\begin{proof}
	Given $F\colon \A_1\to\A_2$ as above, let $\alpha$ be such that $\A_1,\A_2,$ and $F$ are $\alpha$-accessible and $F((\A_1)_{\alpha})\subseteq (\A_2)_{\alpha}$ (see~\ref{preserving-presentables} and~\ref{flat-to-conical}). Then $F$ is the left Kan extension of its restriction to $(\A_1)_\alpha$ and we can consider the commutative (up to isomorphism) diagram below
	\begin{center}
		\begin{tikzpicture}[baseline=(current  bounding  box.south), scale=2]
			
			\node (a0) at (0,0.7) {$[\C,\V]$};
			\node (b0) at (1.1,0.7) {$[\B_2,\V]$};
			\node (c0) at (0,0) {$\A_1$};
			\node (d0) at (1.1,0) {$\A_2$};
			\node (e0) at (0,-0.7) {$\C^{op}$};
			\node (f0) at (1.1,-0.7) {$\B_2^{op}$};
			
			\path[font=\scriptsize]
			
			(a0) edge [->] node [above] {$\tx{Lan}_G$} (b0)
			(c0) edge [right hook->] node [left] {} (a0)
			(d0) edge [right hook->] node [right] {} (b0)
			(c0) edge [->] node [below] {$F$} (d0)
			(e0) edge [right hook->] node [left] {} (c0)
			(f0) edge [right hook->] node [right] {} (d0)
			(e0) edge [->] node [below] {$G^{op}$} (f0);
		\end{tikzpicture}	
	\end{center}
	where $\C=(\A_1)_{\alpha}^{op}$, $\B_2=(\A_2)_\alpha^{op}$, and $G$ is the restriction of $F^{op}$ to $\C$.
	
	Now consider the $\V$-functor $L\colon \B_2\to [\C^{op},\V]$ sending $B$ to $LB=\B_2(G-,B)$ (note that $L^{op}$ is the virtual left adjoint to $G^{op}$). Let $\B_1$ be the full subcategory of $[\C^{op},\V]$ spanned by the representables and the essential image of $L$; then we have a fully faithful inclusion $H\colon \C\hookrightarrow\B_1$ and a $\V$-functor $K\colon \B_2\to\B_1$ induced by $L$. \\
	The next step is to check that the triangle below commutes up to isomorphism.
	\begin{center}
		\begin{tikzpicture}[baseline=(current  bounding  box.south), scale=2]
			
			\node (a) at (0.6,0.5) {$[\B_1,\V]$};
			\node (c) at (0, 0) {$[\C,\V]$};
			\node (d) at (1.2, 0) {$[\B_2,\V]$};
			
			\path[font=\scriptsize]
			
			(c) edge [right hook->] node [above] {$\tx{Lan}_H\ \ \ \ \ \ \ \ \ $} (a)
			(a) edge [->] node [above] {$\ \ \ \ \ \ \ \ \ -\circ K$} (d)
			(c) edge [->] node [below] {$\tx{Lan}_G$} (d);
			
		\end{tikzpicture}	
	\end{center} 
	Since all the $\V$-functors involved are cocontinuous it's enough to prove that the isomorphism holds for all representables in $[\C,\V]$. Consider $C\in\C$: on one hand $\tx{Lan}_G(\C(C,-))\cong \B_2(GC,-)$; on the other $\tx{Lan}_H(\C(C,-))\cong \B_1(HC,-)$ which by precomposition with $K$ gives 
	$$\B_1(HC,K-)\cong [\C^{op},\V](YC,L-)\cong L(-)(C)\cong \B_2(GC,-) $$
	as desired. 
	
	Gluing this triangle with the diagram above we find the desired square in $(2)$. To conclude it's then enough to note that $\A_2$ is identified with the full subcategory of $[\B_2,\V]$ spanned by the $\alpha$-flat functors out of $\B_2$, and hence is the category of models of a sketch. The same holds for the inclusion of $\A_1$ in $[\C,\V]$; moreover $[\C,\V]$ itself is the category of models for a sketch on $\B_1$ (because the inclusion $H\colon \C\hookrightarrow\B_1$ is dense and therefore the essential image of $\tx{Lan}_H$ is given by a full subcategory of functors preserving specified colimits). It follows then that $\A_1$ is the category of models for a sketch in $\B_1$. 
\end{proof}

In the theorem below we say that a $\V$-functor is an isofibration if its underlying ordinary functor is one.

\begin{teo}\label{limitacc}
	The 2-category $\V\tx{-}\bo{Acc}$ has all flexible (and hence all pseudo- and bi-) limits, as well as all pullbacks along isofibrations, and the forgetful functors $U\colon \V\tx{-}\bo{Acc}\to \V\tx{-}\bo{CAT}$ and $(-)_0\colon  \V\tx{-}\bo{Acc}\to \bo{Acc}$ preserve them.
\end{teo}
\begin{proof}
	We prove this using the fact that a 2-category has all flexible limits if and only if it has products, inserters, equifiers, and splitting of idempotent equivalences (see \cite[Theorem~4.9]{BKPS89:articolo}); moreover the latter comes for free in $\V\tx{-}\bo{Acc}$ thanks to \cite[Remark~7.6]{BKPS89:articolo}.
	
	{\em $(a)$ Pullbacks along isofibrations}. Let $F_1\colon \A_1\to\K$ and $F_2\colon \A_2\to\K$ be accessible $\V$-functors between accessible $\V$-categories, and assume that $F_1$ is an isofibration. Consider the pullback $\A_{12}$ of this pair in $\V\tx{-}\bo{CAT}$, with projections $P_1\colon \A_{12}\to\A_1$ and $P_2\colon \A_{12}\to\A_2$, and note that, since $F_1$ is an isofibration, this can be seen as a bipullback in $\V\tx{-}\bo{CAT}$. By Lemma~\ref{fullness-of-sketches} we can find small $\V$-categories $\B_1,\B_2$ and $\C$ together with $\V$-functors $K_i\colon \C\to\B_1$ for which: $\A_i=\tx{Mod}(\S_i)$, $\K=\tx{Mod}(\T)$ for some sketches $\S_i=(\B_i,\mathbb{L}_i,\mathbb{C}_i)$ and $\T=(\B,\mathbb{L},\mathbb{C})$, and such that $-\circ K_i$ restricts to $F_i$ (note that $\C$ can be chosen to be the same for both $\A_1$ and $\A_2$ thanks to the final assertion in the Lemma).
	Let $\B_{12}$ be the pushout of $K_1$ and $K_2$ in $\V\tx{-}\bo{CAT}$, with maps $J_i\colon \B_i\to\B_{12}$. Then $\B_{12}$ is sent to a pullback through $[-,\V]$ providing the commutative cube below
	\begin{center}
		\begin{tikzpicture}[baseline=(current  bounding  box.south), scale=2, on top/.style={preaction={draw=white,-,line width=#1}}, on top/.default=6pt]
			
			\node (a0) at (0,1.1) {$\A_{12}$};
			\node (b0) at (1.3,1.1) {$\A_2$};
			\node (c0) at (0,0) {$\A_1$};
			\node (d0) at (1.3,0) {$\K$};
			\node (e0) at (0.15,0.85) {$\lrcorner$};
			
			\node (a0') at (0.7,0.6) {$[\B_{12},\V]$};
			\node (b0') at (2,0.6) {$[\B_2,\V]$};
			\node (c0') at (0.7,-0.5) {$[\B_1,\V]$};
			\node (d0') at (2,-0.5) {$[\C,\V]$};
			\node (e0') at (0.9,0.35) {$\lrcorner$};
			
			\path[font=\scriptsize]

			(a0') edge [->] node [above] {} (b0')
			(a0') edge [->] node [left] {} (c0')
			(b0') edge [->] node [right] {$-\circ K_2$} (d0')
			(c0') edge [->] node [below] {$-\circ K_1$} (d0')
			
			(a0) edge [->] node [above] {$P_2$} (b0)
			(a0) edge [->] node [left] {$P_1$} (c0)
			(b0) edge [->, on top] node [right] {} (d0)
			(c0) edge [->, on top] node [below] {} (d0)
			
			(a0) edge [right hook->] node [above] {} (a0')
			(b0) edge [right hook->] node [above] {} (b0')
			(c0) edge [right hook->] node [above] {} (c0')
			(d0) edge [right hook->] node [above] {} (d0');
		\end{tikzpicture}	
	\end{center}
	where the $\V$-functors not labelled are $F_1,F_2$, the precomposition functors $-\circ J_1,-\circ J_2$, and the inclusions. It follows that $\A_{12}$ can be identified with the full subcategory fo $[\B_{12},\V]$ whose projections to $[\B_i,\V]$ land in $\A_i$.\\
	Consider now the sketch $\S_{12}=(\B_{12},\mathbb{L}_{12},\mathbb{C}_{12})$ defined by 
	$$ \mathbb{L}_{12}:=\{ J_i\eta\ |\ i=1,2, \eta\in\mathbb{L}_i \},\ \  \mathbb{C}_{12}:=\{ J_i\mu\ |\ i=1,2, \eta\in\mathbb{C}_i \}.$$ 
	It follows at once that $F\colon \B_{12}\to\V$ is a model of $\S_{12}$ if and only if $F\circ J_i$ is a model of $\S_i$, for $i=1,2$; therefore $\A\simeq\tx{Mod}(\S_{12})$ is accessible.\\
	To conclude, consider $\alpha$ such that $\A_1$, $\A_2$, and $\K$ have, and $F_1,F_2$ preserve, all $\alpha$-flat colimits; then it's a standard argument to check that $\A_{12}$ has $\alpha$-flat colimits as well and $P_1$ and $P_2$ preserve them (this follows at once from the fact we are dealing with a bipullback, plus that the homs in $\A_{12}$ come as pullbacks of homs in $\A_1$,$\A_2$, and $\K$, and that $\alpha$-flat colimits commute with them in $\V$). It's now routine to check that $\A_{12}$ is actually a pullback in $\V\tx{-}\bo{Acc}$.
	
	{\em $(b)$ Products}. Let $(\A_i)_{i\in I}$ be a small family of accessible $\V$-categories and denote by $\A=\textstyle\prod_i\A_i$ the product in $\V\tx{-}\bo{CAT}$. For each $i\in I$ consider a sketch $\S_i=(\B_i,\mathbb{L}_i,\mathbb{C}_i)$ for which $\A_i\simeq\tx{Mod}(\S_i)$. Then we can see $\A$ as a full subcategory of $$\prod_{i\in I}[\B_i,\V]\cong [\B,\V]$$
	with $\B:=\textstyle\coprod_{i\in I}\B_i$ in $\V\tx{-}\bo{CAT}$. Now, for each $i\in I$, we can consider the coproduct inclusion $J_i\colon \B_i\to\B$ and define the set of cylinders
	$$ \mathbb{L}:=\{ J_i\eta\ |\ i\in I, \eta\in\mathbb{L}_i \} $$ 
	and of cocylinders
	$$ \mathbb{C}:=\{ J_i\mu\ |\ i\in I, \eta\in\mathbb{C}_i \}, $$ 
	which identify a sketch $\S=(\B,\mathbb{L},\mathbb{C})$. It follows then by construction that $F\colon \B\to\V$ lies in $\A$ if and only if its components $F_i$ are model of $\S_i$ for each $i\in I$, if and only if $F$ is a models of $\S$; thus $\A\simeq\tx{Mod}(\S)$ is accessible. Moreover considering $\alpha$ such that each $\A_i$ is $\alpha$-accessible, an easy calculation shows then that $\A$ has all $\alpha$-flat colimits and these are preserved by the projections; it now easily follows that $\B$ is the product of $(\A_i)_{i\in I}$ in $\V\tx{-}\bo{Acc}$.
	
	{\em $(c)$ Inserters}. Let $F,G\colon \A\to\K$ be a parallel pair in $\V\tx{-}\bo{Acc}$; then their equifier $\B$ can be seen as the pullback 
	\begin{center}
		\begin{tikzpicture}[baseline=(current  bounding  box.south), scale=2]
			
			\node (a0) at (0,0.8) {$\B$};
			\node (b0) at (0.9,0.8) {$\A$};
			\node (c0) at (0,0) {$\K^\mathbbm{2}$};
			\node (d0) at (0.9,0) {$\K\times\K$};
			\node (e0) at (0.15,0.65) {$\lrcorner$};
			
			\path[font=\scriptsize]
			
			(a0) edge [->] node [above] {} (b0)
			(a0) edge [->] node [left] {} (c0)
			(b0) edge [->] node [right] {$(F,G)$} (d0)
			(c0) edge [->] node [below] {$\pi$} (d0);
		\end{tikzpicture}	
	\end{center}
	where $\pi$ is the projection induced by the inclusion $2\to\mathbbm{2}$. All the $\V$-categories involved are accessible (by the corollary above), as well as the $\V$-functors, and $\pi$ is an isofibration. So the result follows from point $(a)$. 
	
	{\em $(d)$ Equifiers}. Let $\mu,\eta\colon F\Rightarrow G\colon \A\to\K$ be a parallel pair of 2-cells in $\V\tx{-}\bo{Acc}$. Arguing as above we can write their equifier $\B$ as the pullback
	\begin{center}
		\begin{tikzpicture}[baseline=(current  bounding  box.south), scale=2]
			
			\node (a0) at (0,0.8) {$\B$};
			\node (b0) at (0.9,0.8) {$\A$};
			\node (c0) at (0,0) {$\K^\mathbbm{2}$};
			\node (d0) at (0.9,0) {$\K^\P$};
			\node (e0) at (0.15,0.65) {$\lrcorner$};
			
			\path[font=\scriptsize]
			
			(a0) edge [->] node [above] {} (b0)
			(a0) edge [->] node [left] {} (c0)
			(b0) edge [->] node [right] {$(\bar{\mu},\bar{\eta})$} (d0)
			(c0) edge [->] node [below] {$\rho$} (d0);
		\end{tikzpicture}	
	\end{center}
	where $\P$ is the free living parallel pair and $\rho$ is the diagonal induced by the projection $\P\to\mathbbm{2}$. Again this exists in $\V\tx{-}\bo{Acc}$ by point $(a)$ since $\rho$ is an isofibration and all the involved $\V$-categories and $\V$-functors are accessible.
	
	The fact that $U\colon \V\tx{-}\bo{Acc}\to \V\tx{-}\bo{CAT}$ preserves the these limits is a direct consequence of the proof (since we took the limits in $\V\tx{-}\bo{CAT}$ and then proved that they are still limits in $\V\tx{-}\bo{Acc})$. Regarding the underlying functor $(-)_0\colon  \V\tx{-}\bo{Acc}\to \bo{Acc}$, this preserves all the limits in question since the forgetful functors $\V\tx{-}\bo{CAT}\to \bo{CAT}$ and $\bo{Acc}\to \bo{CAT}$ do.
\end{proof}

\begin{obs}
	In the proof we show that, in addition to flexible limits, $\V\tx{-}\bo{Acc}$ has pullbacks along isofibrations; this shouldn't be too surprising in light of \cite[Proposition~A.1]{Bou2021:articolo}.
\end{obs}

Similarly one obtains the same result for accessible $\V$-categories with limits of some class $\Psi$. Let $\V\tx{-}\bo{Acc}_\Psi$ be the 2-category of the accessible $\V$-categories with $\Psi$-limits, $\Psi$-continuous and accessible $\V$-functors, and $\V$-natural transformations. 

\begin{cor}\label{limitsofpsiacc}
	The 2-category $\V\tx{-}\bo{Acc}_\Psi$ has all flexible (and hence all pseudo- and bi-) limits and the forgetful functor $\V\tx{-}\bo{Acc}_\Psi\to \V\tx{-}\bo{Acc}$ preserves them.
\end{cor}
\begin{proof}
	This is a direct consequence of the results above since the following 
	\begin{center}
		\begin{tikzpicture}[baseline=(current  bounding  box.south), scale=2]
			
			\node (a0) at (0,0.8) {$\V\tx{-}\bo{Acc}_\Psi$};
			\node (b0) at (1.2,0.8) {$\V\tx{-}\bo{CAT}_\Psi$};
			\node (c0) at (0,0) {$\V\tx{-}\bo{Acc}$};
			\node (d0) at (1.2,0) {$\V\tx{-}\bo{CAT}$};
			\node (e0) at (0.25,0.6) {$\lrcorner$};
			
			\path[font=\scriptsize]
			
			(a0) edge [->] node [above] {} (b0)
			(a0) edge [->] node [left] {} (c0)
			(b0) edge [->] node [right] {} (d0)
			(c0) edge [->] node [below] {} (d0);
		\end{tikzpicture}	
	\end{center}
	is a pullback, where $\V\tx{-}\bo{CAT}_\Psi$ is the 2-category of $\Psi$-complete $\V$-categories, $\Psi$-continuous $\V$-functors, and $\V$-natural transformations; this has all flexible limits and the forgetful functor $\V\tx{-}\bo{CAT}_\Psi\to \V\tx{-}\bo{CAT}$ preserves them (see \cite{Bir84:tesi}).
\end{proof}

In particular, taking $\Psi$ to be the class of all small limits, $\V\tx{-}\bo{Acc}_\Psi$ can be identified with the 2-category $\V\tx{-}\bo{Lp}$ of locally presentable $\V$-categories, accessible $\V$-functors with a left adjoint, and $\V$-natural transformations. Thus we recover part of \cite[Theorem~6.10]{Bir84:tesi}:

\begin{cor}
	The 2-category $\V\tx{-}\bo{Lp}$ has all flexible (and hence all pseudo- and bi-) limits and the forgetful functor $\V\tx{-}\bo{Lp}\to \V\tx{-}\bo{CAT}$ preserves them.
\end{cor}
 
\vspace{6pt}

\subsection{The conically accessible case}

Analogous results hold for conical accessibility. Denote by $\V\tx{-}\bo{cAcc}$ the 2-category of conically accessible $\V$-categories, conically accessible $\V$-functors, and $\V$-natural transformations. Note that $\V\tx{-}\bo{Acc}$ sits in $\V\tx{-}\bo{cAcc}$ as a full subcategory by Corollary~\ref{flat-to-conical} and Proposition~\ref{accessible-functors}.

In the proof of the following theorem we rely on the facts that $\bo{Acc}$ has all flexible limits (which can be seen of course as a consequence of the previous theorem for $\V=\bo{Set}$) and that $(-)_0\colon \V\tx{-}\bo{Acc}\to \bo{Acc}$ preserves them.

\begin{teo}
	The 2-category $\V\tx{-}\bo{cAcc}$ has all flexible (and hence all pseudo- and bi-) limits and the forgetful functors $U_c\colon \V\tx{-}\bo{cAcc}\to \V\tx{-}\bo{CAT}$, $(-)_0\colon  \V\tx{-}\bo{cAcc}\to \bo{Acc}$, and $J\colon \V\tx{-}\bo{Acc}\hookrightarrow \V\tx{-}\bo{cAcc} $ preserve them.
\end{teo}
\begin{proof}
	(Powers) Let $\A$ be a conically accessible $\V$-category and $\C$ be an ordinary small category. Consider the power $\A^\C=[\C_\V,\A]$ in $\V\tx{-}\bo{CAT}$. Let $\alpha$ be such that $\A$ has conical $\alpha$-filtered colimits; then $\A^\C$ has them as well since they can be computed pointwise in $\A$. Moreover, given $X\in\A^\C$, we can consider $\beta$ such that each $X(c)$ is conically $\beta$-presentable in $\A$ and the end 
	$$ \A^\C(X,-)\cong\int_{c\in\C}\A(X(c),\tx{ev}_c-) $$
	is $\beta$-small, so that $\A^\C(X,-)$ preserves all $\beta$-filtered colimits. It follows then that each object of $\A^\C$ is conically presentable. Now $(\A^\C)_0\cong[\C,\A_0]$ is accessible since $\A_0$ is; thus the conical accessibility of $\A^\C$ follows from Proposition~\ref{conical-flat}. The fact that $\A^\C$ is the desired power in $\V\tx{-}\bo{cAcc}$ follows from the fact that a $\V$-functor $\B\to\A^\C$ is conically accessible if and only if its transpose ordinary functor $\C\to\V\tx{-}\bo{CAT}(\B,\A)$ lands in $\V\tx{-}\bo{cAcc}(\B,\A)$.
	
	(Pullbacks along isofibrations) Let $F\colon \A\to\K$ and $G\colon \B\to\K$ be conically accessible $\V$-functors between conically accessible $\V$-categories, and assume that $F$ is an isofibration. Consider the pullback $\C$ of this in $\V\tx{-}\bo{CAT}$, with projections $P\colon \C\to\A$ and $Q\colon \C\to\B$, and note that, since $F$ is an isofibration, this can be seen as a pseudo-pullback. Let $\alpha$ be such that $\A$, $\B$, and $\K$ have, and $F$,$G$ preserve, all conical $\alpha$-filtered colimits; then it's a standard argument to check that $\C$ has conical $\alpha$-filtered colimits as well and $P$ and $Q$ preserve them. Similarly each object of $\C$ is conically presentable (here we use that each object of $\A$,$\B$, and $\K$ is). Moreover $\C_0$ is accessible being a pseudo-pullback of ordinary accessible categories. In conclusion $\C$ is conically accessible by Proposition~\ref{conical-flat}, and it's now routine to check that it is actually a pseudo-pullback in $\V\tx{-}\bo{Acc}$.
	
	The proof for products is very similar. Let $(\A_i)_{i\in I}$ be a small family of accessible $\V$-categories and denote by $\A=\textstyle\prod_i\A_i$ their product in $\V\tx{-}\bo{CAT}$. Consider $\alpha$ such that each $\A_i$ has all conical $\alpha$-filtered colimits, an easy calculation shows then that $\A$ has all conical $\alpha$-filtered colimits as well and these are preserved by the projections. Moreover, given any $A=(A_i)_i\in\A$, consider $\beta$ such that $I$ is $\beta$-small and each $A_i$ is conically $\beta$-presentable in $\A_i$; then $A$ is conically $\beta$-presentable in $\A$. Finally $\A$ is conically accessible thanks to Proposition~\ref{conical-flat} since $\A_0\cong\textstyle\prod_i(\A_i)_0$ is an ordinary accessible category. It now easily follows that $\A$ is the product of $(\A_i)_{i\in I}$ in $\V\tx{-}\bo{cAcc}$.
	
	Inserters and equifiers can be obtained from powers and pullbacks along isofibrations (see the proof of the previous theorem). Finally, splittings of idempotent equivalences come again for free thanks to \cite[Remark~7.6]{BKPS89:articolo}.
\end{proof}

We can consider also the conical version of Corollary~\ref{limitsofpsiacc}. Let $\V\tx{-}\bo{cAcc}_\Psi$ be the 2-category of the conically accessible $\V$-categories with $\Psi$-limits, $\Psi$-continuous and conically accessible $\V$-functors, and $\V$-natural transformations. The same argument then gives:

\begin{cor}
	The 2-category $\V\tx{-}\bo{cAcc}_\Psi$ has all flexible (and hence all pseudo- and bi-) limits and the forgetful functor $\V\tx{-}\bo{cAcc}_\Psi\to \V\tx{-}\bo{cAcc}$ preserves them.
\end{cor}

\vspace{6pt}

\subsection{Sketches over a base}

\begin{Def}
	Let $\S=(\B,\mathbb{L},\mathbb{C})$ be a sketch and $\A$ be a $\V$-category. A {\em model of a $\S$ in $\A$} is a $\V$-functor $F\colon \B\to\A$ which transforms each cylinder of $\mathbb{L}$ into a limit cylinder in $\A$, and each cocylinder of $\mathbb{C}$ into a colimit cocylinder in $\A$. We denote by $\tx{Mod}(\S,\A)$ the full subcategory of $[\B,\A]$ spanned by the models of $\S$ in $\A$.
\end{Def}

The accessibility of $\tx{Mod}(\S,\A)$ when $\A$ is an accessible category depends on set theory in general (even in the ordinary case, see \cite[Example~A.19]{AR94:libro}), but when $\A$ has enough colimits something can be said, generalizing \cite[Theorem~2.60]{AR94:libro}.

\begin{prop}\label{sketchoveracc}
	Let $\S=(\B,\mathbb{L},\mathbb{C})$ be a sketch and $\A$ be an accessible $\V$-category with $M$-colimits for any weight $M$ appearing in $\mathbb{C}$. Then $\tx{Mod}(\S,\A)$ is an accessible $\V$-category.
\end{prop}
\begin{proof}
	Let $\S_\mathbb{L}=(\B,\mathbb{L})$ and $\S_\mathbb{C}=(\B,\mathbb{C})$ be the limit and colimit parts of $\S$. Then $\tx{Mod}(\S,\A)=\tx{Mod}(\S_\mathbb{L},\A)\cap\tx{Mod}(\S_\mathbb{C},\A)$ and it's enough to prove that these two $\V$-categories are accessible.
	
	Regarding the limit case, let $\C=\A_\alpha^{op}$, with inclusion $J\colon\C\to\A$, and consider the sketch $\S'=(\C\otimes\B,\mathbb{L}_\C)$ defined as in the proof of Proposition~\ref{acc-funtors} starting from $\S_\mathbb{L}$. Then $\tx{Mod}(\S_\mathbb{L},\A)$ can be seen as the intersection
	\begin{center}
		\begin{tikzpicture}[baseline=(current  bounding  box.south), scale=2]
			
			\node (a0) at (0,0.8) {$\tx{Mod}(\S_\mathbb{L},\A)$};
			\node (b0) at (1.3,0.8) {$\tx{Mod}(\S')$};
			\node (c0) at (0,0) {$[\B,\A]$};
			\node (d0) at (1.3,0) {$[\C\otimes\B,\A]$};
			\node (e0) at (0.3,0.55) {$\lrcorner$};
			
			\path[font=\scriptsize]
			
			(a0) edge [right hook->] node [above] {} (b0)
			(a0) edge [right hook->] node [left] {} (c0)
			(b0) edge [right hook->] node [right] {} (d0)
			(c0) edge [right hook->] node [below] {$K$} (d0);
		\end{tikzpicture}	
	\end{center}
	where $K$ is the composite of the inclusion $[\B,\A(J,1)]\colon[\B,\A]\to[\B,[\C,\A]]$ with the isomorphism $[\B,[\C,\A]]\cong [\C\otimes\B,\A]$. Since $[\B,\A]$ and $\tx{Mod}(\S')$ are accessible and accessibly embedded in $[\C\otimes\B,\A]$ it follows that $\tx{Mod}(\S_\mathbb{L},\A)$ is accessible as well.
	
	About the colimit case, for each cocylinder $(\eta:M\Rightarrow\B(H-,B))\in\mathbb{C}$ consider the $\V$-functor $G_\eta\colon[\B,\A]\to\A^{\mathbbm{2}}$ such that $G_\eta(F)\colon M*FH\to FB$ is the unique morphism induced by the cocylinder $F\eta$, and acts on hom-objects accordingly (to define $G_\eta$ we are using that $\A$ has $M$-colimits). Now consider the full subcategory $\A^{\cong}$ of $\A^{\mathbbm{2}}$ spanned by the isomorphisms of $\A$, this is accessible since it is equivalent to $\A$. Then we can see $\tx{Mod}(\S_\mathbb{C},\A)$ as the pullback below
	\begin{center}
		\begin{tikzpicture}[baseline=(current  bounding  box.south), scale=2]
			
			\node (a0) at (0,0.9) {$\tx{Mod}(\S_\mathbb{C},\A)$};
			\node (b0) at (1.4,0.83) {$\prod_{\eta\in\mathbb{C}}\A^{\cong}$};
			\node (c0) at (0,0) {$[\B,\A]$};
			\node (d0) at (1.4,-0.07) {$\prod_{\eta\in\mathbb{C}}\A^{\mathbbm{2}}$};
			\node (e0) at (0.3,0.65) {$\lrcorner$};
			
			\path[font=\scriptsize]
			
			(a0) edge [->] node [above] {} ([yshift=2]b0.west)
			(a0) edge [right hook->] node [left] {} (c0)
			([yshift=2]b0.south) edge [right hook->] node [right] {} ([yshift=-1.3]d0.north)
			(c0) edge [->] node [below] {$(G_\eta)_{\eta\in\mathbb{C}}$} ([yshift=2]d0.west);
		\end{tikzpicture}	
	\end{center}
	where the right vertical arrow is an isofibration. The three $\V$-categories involved in the limit are accessible by Theorem~\ref{limitacc} and the $\V$-functors are easily seen to be accessible; thus $\tx{Mod}(\S_\mathbb{C},\A)$ is an accessible $\V$-category as well again by Theorem~\ref{limitacc}.
\end{proof}

Immediate consequences are:

\begin{cor}
	For any accessible $\V$-category $\A$ and any limit sketch $\S$ the $\V$-category $\tx{Mod}(\S,\A)$ is accessible.	
\end{cor}

\begin{cor}
	For any locally presentable $\V$-category $\K$ and any sketch $\S$ the $\V$-category $\tx{Mod}(\S,\K)$ is accessible.	
\end{cor}

As in the accessible case we can consider models of sketches over a conically accessible $\V$-category:

\begin{prop}
	Let $\S=(\B,\mathbb{L},\mathbb{C})$ be a sketch and $\A$ be a conically accessible $\V$-category with $M$-colimits for any weight $M$ appearing in $\mathbb{C}$. Then $\tx{Mod}(\S,\A)$ is a conically accessible $\V$-category.
\end{prop}
\begin{proof}
	Same as that of Proposition~\ref{sketchoveracc}.
\end{proof}

\begin{cor}
	For any conically accessible $\V$-category $\A$ and any limit sketch $\S$ the $\V$-category $\tx{Mod}(\S,\A)$ is conically accessible.	
\end{cor}


\end{document}